\documentclass[a4paper]{amsart}
\usepackage[utf8]{inputenc}
\usepackage[english]{babel}

\usepackage{amscd,amssymb}
\usepackage{amsmath}
\usepackage{amsthm}
\usepackage{graphicx}
\usepackage{fancybox}
\usepackage{epic,eepic}
\usepackage{amstext}
\usepackage{mathtools}
\usepackage{esint}
\usepackage[square,numbers]{natbib}
\usepackage{booktabs}
\usepackage[hide links]{hyperref}
\usepackage{comment}
\usepackage{mathtools}
\usepackage{tikz,pgfplots}
\usepackage{subcaption}

\usepackage[noabbrev, capitalise, nameinlink]{cleveref}
\crefname{equation}{}{}
\crefname{assumption}{Assumption}{Assumptions}
\crefformat{equation}{\textup{#2(#1)#3}}

\newtheorem{theorem}{Theorem}[section]

\newtheorem{lemma}[theorem]{Lemma}

\theoremstyle{definition}

\theoremstyle{remark}
\newtheorem{remark}[theorem]{Remark}
\numberwithin{theorem}{section}
\numberwithin{equation}{section}
\numberwithin{figure}{section}

\newcommand{\fraka}{\mathfrak{a}}
\newcommand{\frakb}{\mathfrak{b}}

\def\R{\mathbb{R}}

\def\XXint#1#2#3{{\setbox0=\hbox{$#1{#2#3}{\int}$ }
\vcenter{\hbox{$#2#3$ }}\kern-.6\wd0}}

\def\Nb{\mathsf{N}}
\def\TH{\mathcal T_H}

\def\fraka{\mathfrak a}
\def\frakb{\mathfrak b}
\def\with{\,:\,}
\def\ds{\,\mathrm{d}s}

\numberwithin{equation}{section}
\numberwithin{theorem}{section}

\AtBeginDocument{%
	\def\MR#1{}
}

\allowdisplaybreaks
\title[A higher-order multiscale method for nondivergence-form PDEs]{A post-processed higher-order multiscale method for nondivergence-form elliptic equations}
\author[M.~Hauck, R.~Maier, T.~Sprekeler]{Moritz Hauck$^*$, Roland Maier$^*$, Timo Sprekeler$^\dagger$}
\address{${}^*$ Institute for Applied and Numerical Mathematics, Karlsruhe Institute of Technology, Englerstr.~2, 76131 Karlsruhe, Germany}
\email{\{moritz.hauck,roland.maier\}@kit.edu}
\address{${}^{\dagger}$ Department of Mathematics, Texas A\&M University, College Station, TX 77843, USA}
\email{timo.sprekeler@tamu.edu}
 \date{\today}

\begin{document}
\begin{abstract}
We study the finite element approximation of linear second-order elliptic partial differential equations in nondivergence form with highly heterogeneous diffusion and drift coefficients. A generalized Cordes condition is imposed to guarantee that a suitably renormalized version of the nondivergence-form differential operator is near the Laplacian. Based on a stabilized symmetric formulation for the gradient that enables the use of $H^1$-conforming approximation spaces, we construct a multiscale method following the methodology of the localized orthogonal decomposition with coarse basis functions tailored to the heterogeneous coefficients. We employ a novel post-processing strategy to obtain higher-order convergence rates, overcoming previous limitations imposed by the low regularity of the load functional. Numerical experiments demonstrate the performance of the method.
\end{abstract}


\maketitle

\section{Introduction}
On a bounded convex polyhedral domain $\Omega \subset \R^d$ in dimension $d\in \{2,3\}$, we consider the numerical approximation of the homogeneous Dirichlet problem
\begin{align}\label{u prob intro}
\left\{
\begin{aligned}
    Lu &= f\quad\text{in }\Omega,\\
    u &= 0\quad\text{on }\partial\Omega,
\end{aligned}
\right.
\end{align}
where $f\in L^2(\Omega)$ is a given source term and $L$ is a linear second-order differential operator in nondivergence form defined by
\begin{align*}
  Lv \coloneqq A:D^2 v  + b\cdot \nabla v = \sum_{i,j=1}^d a_{ij} \,\partial^2_{ij} v + \sum_{k=1}^d b_k\,\partial_k v,
\end{align*}
involving a uniformly elliptic diffusion coefficient $A = (a_{ij})_{1\leq i,j\leq d}\in L^{\infty}(\Omega;\R^{d\times d}_{\mathrm{sym}})$ and a drift coefficient $b= (b_k)_{1\leq k\leq d}\in L^{\infty}(\Omega;\R^d)$. Our focus is on regimes where the coefficients are multiscale, exhibiting strong oscillations across multiple, nonseparated scales. Motivated by the fact that, in the absence of any further structural assumptions, the map $L\colon H^2(\Omega)\cap H^1_0(\Omega)\rightarrow L^2(\Omega)$ is not bijective in general, we impose the following Cordes-type condition:
\begin{align}\label{C-type}
    \exists\, \delta\in \left( \frac{\eta}{1+C_{\mathrm{P}}^2}\,,\,1\right]\colon\quad \frac{A:A+b\cdot b}{(\mathrm{tr}(A))^2} \leq \frac{1}{d-1+\delta}\quad\text{a.e. in }\Omega,
\end{align}
where the constant $C_{\mathrm{P}}\coloneqq  \inf_{\varphi\in V\backslash\{0\}} \|D\varphi\|_{L^2(\Omega)}/\|\varphi\|_{L^2(\Omega)}>0$ is the optimal Poincar\'{e} constant for the subspace $V\coloneqq H^1_t(\Omega;\R^d)$ of $H^1(\Omega;\R^d)$ consisting of functions with vanishing tangential trace on $\partial\Omega$, and the constant $\eta\in \{0,1\}$ is given by $\eta \coloneqq 0$ if $b = 0$ almost everywhere in $\Omega$, and $\eta \coloneqq 1$ otherwise. 

Condition \eqref{C-type} generalizes the classical Cordes condition for operators of the form $v\mapsto A:D^2v$ (see \cite{Cor56}), and is inspired by the Cordes-type condition
\begin{align}\label{Cordes Abc intro}
\exists\, (\epsilon,l)\in \left(0,1\right]\times (0,\infty):\qquad\frac{A:A + \frac{1}{2l}(b\cdot b) + \frac{1}{l^2}c^2}{(\mathrm{tr}(A)+\frac{1}{l}c)^2} \leq \frac{1}{d+\epsilon}\quad\text{a.e. in }\Omega
\end{align} 
for operators of the form $v\mapsto A:D^2 v+b\cdot \nabla v- cv$ (see \cite{SS14}). The need for a new condition comes from the observation that \eqref{Cordes Abc intro} cannot be satisfied when $c = 0$ almost everywhere in $\Omega$. In terms of Campanato's notion of near operators \cite{Cam94}, the Cordes-type condition \eqref{C-type} ensures that a suitably renormalized version of $L$ is near the Laplacian, which in turn implies existence and uniqueness of a strong solution $u\in H^2(\Omega)\cap H^1_0(\Omega)$ to problem \eqref{u prob intro}.

The goal of this work is the construction and rigorous error analysis of a method for the accurate numerical approximation of the multiscale problem \eqref{u prob intro}, a task often referred to as numerical homogenization. We follow the methodology of localized orthogonal decomposition (LOD) \cite{MalP14,Mlqvist2020,Altmann2021}, using the observation that $z = \nabla u$ is the unique element in $V$ such that
\begin{align}\label{z prob intro}
  \forall \varphi\in V\colon \qquad  (\tilde L z,\tilde L \varphi)_{L^2(\Omega)} + \sigma (\mathrm{rot}(z),\mathrm{rot}(\varphi))_{L^2(\Omega)} = (f,\tilde L \varphi)_{L^2(\Omega)}
\end{align}
with $\tilde L \varphi \coloneqq A:D\varphi + b\cdot \varphi$ for $\varphi\in V$, and $\sigma>0$ a chosen constant. Inspired by~\cite{Gal17b,Gal19}, this formulation has the important advantage that it is symmetric and allows for a discretization based on classical $H^1$-conforming finite element spaces. In particular, we can avoid the more technical implementation of $H^2$-conforming discretizations; see, e.g., \cite{KM19}. Once a practical multiscale approximation of $z$ has been obtained, the function $u$ can be easily recovered by a computationally cheap finite element approximation of a Poisson problem on a coarse mesh.

For recent developments in homogenization of nondivergence-form equations, we refer the reader to \cite{GTY20,ST21,Spr24,GST25} for periodic homogenization and \cite{AFL22,GT24,GT25,GSTT25} for stochastic homogenization. Various numerical schemes for the approximation of effective diffusion matrices in periodic homogenization of nondivergence-form equations have been proposed in recent years; see \cite{CSS20,SSZ25} for finite element methods, \cite{FO09} for a finite difference method, and \cite{SWZ25} for a Lagrangian method. Linear elliptic equations in nondivergence-form arise as linearizations of fully nonlinear second-order Hamilton--Jacobi--Bellman and Isaacs equations. Numerical schemes for the approximation of effective Hamiltonians arising in the periodic homogenization of such problems have recently been developed; see \cite{GSS21,KS22,QST24} for finite element methods and \cite{CM09,FO18,FOO18} for finite difference methods. 

The sole reference for the numerical homogenization of nondivergence-form equations with arbitrarily rough coefficients beyond periodicity and scale separation is the $H^2$-conforming LOD method proposed in \cite{FGP24}, which is based on a nonsymmetric formulation and provides an error bound in the $H^1$-norm only. In contrast, the method developed in this work does not require $H^2$-conforming elements, is based on a symmetric formulation, allows for higher-order approximation, and admits provable error bounds for the approximation of second-order derivatives of $u$ in the $L^2$-norm.
We employ the localization strategy proposed in \cite{Hauck2026}, eliminating the undesirable numerical effects observed in \cite{FGP24} that degrade approximation quality for fixed localization parameters. Further, we employ a post-processing step, which enables the extraction of high-order convergence rates in the spirit of higher-order LOD approaches; see~\cite{Maier2021,Dong2023}. We incorporate ideas from the collocation version of the Super-LOD method introduced in \cite{SLOD} as well as strategies for additional corrections as analyzed in~\cite{HellmanMalqvist2019,KalKMW25}. 
Similar constructions for the functionals that implicitly define the LOD approximation space, often called quantities of interest, have also been investigated in \cite{HHM16, Hauck2025HOStokes, HL25, HLS26} for mixed formulations of divergence-form problems.

This paper is structured as follows.
In \cref{Sec: 2}, we show the existence and uniqueness of a strong solution $u$ to \eqref{u prob intro} under the Cordes-type condition \eqref{C-type}, analyze the variational problem \eqref{z prob intro} for $\nabla u$, and discuss the recovery of $u$. In \cref{Sec: 3}, we introduce and rigorously analyze an ideal numerical homogenization method that achieves optimal-order approximations without pre-asymptotic effects. In \cref{sec:decay}, we prove exponential decay properties for a problem-adapted orthogonal projection $\mathcal R$ onto the ideal multiscale space, as well as of related operators. These results are then used in \cref{sec:localization} to localize the operator $\mathcal R$. In  \cref{sec:pracmethod}, we introduce and rigorously analyze a practical numerical homogenization method based on locally computed basis functions, justified by the exponential decay properties and localization obtained in the previous sections. Finally, in \cref{sec:numexp}, we provide various numerical experiments that illustrate the theoretical results.

\subsection*{Notation} 
Throughout this work, we use the notation \( a \lesssim b \) (respectively \( a \gtrsim b \)) to indicate that \( a \leq C\, b \) (respectively \( a \geq C\, b \)), where \( C > 0 \) denotes a generic constant independent of the coarse mesh size \( H \), the localization parameter \( \ell \), and the oscillatory behavior of the  solution~\( u \). The constant \( C \) may depend on the mesh regularity, the spatial dimension \( d \), the ellipticity constants \( \nu_1, \nu_2 \), the $L^\infty(\Omega)$-bound of $b$, the Cordes parameter $\delta$, the stabilization parameter $\sigma$, and the post-processing order \( p \).
We employ standard notation for Lebesgue and Sobolev spaces and their associated norms. For norms defined on the entire domain \( \Omega \), such as the \( L^p(\Omega) \)-norm, we write \( \|\cdot\|_{L^p} \), omitting the domain from the notation. The domain is specified explicitly only when referring to proper subdomains of \( \Omega \).
The gradient and the Hessian of a scalar-valued function $u$ are denoted by $\nabla u$ and $D^2 u$, respectively. The Jacobian of a vector-valued function $v$ is denoted by $Dv$.
Finally, we write $x\cdot y\coloneqq x^{\mathrm{T}}y$ for $x,y\in \R^d$, and $X:Y \coloneqq \mathrm{tr}(X^{\mathrm{T}}Y)$ for $X,Y\in \R^{d\times d}$. 

\section{Model problem}\label{Sec: 2}

Let $\Omega \subset \R^d$ be a bounded convex domain in dimension $d\in \{2,3\}$. We consider the elliptic nondivergence-form problem
\begin{align}\label{u prob}
\left\{
\begin{aligned}
    A:D^2 u + b\cdot \nabla u &= f\quad\text{in }\Omega,\\
    u &= 0\quad\text{on }\partial\Omega.
\end{aligned}
\right.
\end{align}
The data are specified as follows: $f \in L^2(\Omega)$ is a source term, $b \in L^{\infty}(\Omega;\mathbb{R}^d)$ is a drift coefficient, and $A \in L^{\infty}(\Omega;\mathbb{R}^{d \times d}_{\mathrm{sym}})$ is a uniformly elliptic diffusion coefficient, i.e.,
\begin{align}\label{unifm ell}
    \exists\,\nu_1,\nu_2>0\text{ s.t. } \forall \xi\in \R^d\backslash\{0\} \colon \quad
    \nu_1 \leq \frac{ (A\xi)\cdot \xi}{\xi\cdot\xi} \leq \nu_2 \; \text{ a.e. in }\Omega.
\end{align} 
Additionally, we assume that the pair $(A,b)$ satisfies the Cordes-type condition
\begin{align}\label{Cordes}
    \exists\, \delta\in (\delta_0,1]\colon\quad \frac{A:A+b\cdot b}{(\mathrm{tr}(A))^2} \leq \frac{1}{d-1+\delta}\quad\text{a.e. in }\Omega.
\end{align}
Here, the lower bound of admissible values for $\delta$ is given by
\begin{align}\label{d0 and eta}
    \delta_0 \coloneqq \frac{\eta}{1+C_\mathrm{P}^2},\qquad\eta\coloneqq \begin{cases} 
1 &\text{ if }\|b\|_{L^{\infty}} \neq 0,\\
0&\text{ otherwise},
    \end{cases}
\end{align}
where $C_\mathrm{P}$ is the Poincar\'{e} constant
\begin{align}\label{Cp}
  C_\mathrm{P}\coloneqq  \inf_{\varphi\in V\backslash\{0\}} \frac{\|D\varphi\|_{L^2}}{\|\varphi\|_{L^2}} > 0
\end{align}
for the subspace $V$ of $H^1(\Omega;\mathbb{R}^d)$ consisting of functions with vanishing tangential trace on $\partial\Omega$, that is,
\begin{align}\label{eq:V}
V \coloneqq H^1_t(\Omega;\R^d) \coloneqq \big\{ \varphi \in H^1(\Omega;\mathbb{R}^d)\  :\  \varphi-(\varphi\cdot n)n = 0 \text{ on } \partial\Omega \big\},
\end{align}
where $n$ denotes the exterior unit normal on $\partial\Omega$; see, e.g., \cite{Bis88}.

It is important to note that no regularity assumptions beyond $L^\infty$ are imposed on the
coefficients $A$ and $b$. The scenario of particular interest here is when the coefficients are rough, with oscillations across multiple length scales.

\subsection{Well-posedness}\label{Sec: 2.1}

In the absence of a drift term \(b\), condition~\eqref{Cordes} reduces to the classical Cordes condition introduced in \cite{Cor56}, which in two dimensions is already implied by the uniform ellipticity condition~\eqref{unifm ell}. In this case, it is well-known (cf.~\cite{Tal65,SS13}) that problem~\eqref{u prob} admits a unique strong solution \(u \in H^2(\Omega)\cap H^1_0(\Omega)\).
When the drift term does not vanish, we can nevertheless establish existence and uniqueness of a strong solution to~\eqref{u prob} under condition~\eqref{Cordes}, using ideas from \cite{SSZ25,Spr26}. To this end, we first introduce the notion of near operators as established by Campanato in  \cite{Cam94}.
\begin{lemma}[Campanato nearness]\label{Lmm: Campanato}
Let $L_1,L_2\colon X \rightarrow Y$ be two maps from a set $X\neq \emptyset$ to a real Banach space $(Y,\|\cdot\|)$. Suppose that $L_2$ is bijective and that $L_1$ is $($Campanato-$)$near $L_2$, that is, there exist constants $c > 0$ and $K\in [0,1)$ such that
\begin{align*}
\forall v_1,v_2\in X:\quad \| L_2(v_1) - L_2(v_2) - c\left[ L_1(v_1) - L_1(v_2)\right]\| \leq K \|L_2(v_1) - L_2(v_2)\|.
\end{align*}
Then, $L_1$ is bijective and we have the bound
\begin{align}\label{Cam bd}
\forall v_1,v_2\in X:\quad  \|L_2(v_1)-L_2(v_2)\| \leq \frac{c}{1-K} \|L_1(v_1)-L_1(v_2)\|.
\end{align}
\end{lemma}
\begin{proof}
    We give a short proof for completeness. Since $L_2$ is bijective, we have that $(X,\rho)$ is a nonempty complete metric space with the metric $\rho:X\times X\rightarrow [0,\infty)$ defined by $\rho(v_1,v_2)\coloneqq \|L_2(v_1)-L_2(v_2)\|$ for $v_1,v_2\in X$. For $f\in Y$, we introduce
    \begin{align*}
        T_f: X\rightarrow X,\qquad T_f(v)\coloneqq L_2^{-1}(L_2(v)+c[f -  L_1(v)]), 
    \end{align*}
    which is well-defined by bijectivity of $L_2$. We claim that $T_f$ is a contraction on $(X,\rho)$ for all $f\in Y$. Indeed, as $L_1$ is near $L_2$, we have for any $f\in Y$ and $v_1,v_2\in X$ that
    \begin{align*}
        \rho(T_f(v_1),T_f(v_2)) = \|L_2(v_1) - L_2(v_2) - c\left[ L_1(v_1) - L_1(v_2)\right] \| \leq K\rho(v_1,v_2).
    \end{align*}
    Banach's fixed point theorem yields that for any $f\in Y$ there exists a unique $u\in X$ such that $T_f(u) = u$, i.e., $L_1(u) = f$. It follows that $L_1$ is bijective. Finally, for any $v_1,v_2\in X$, we use the triangle inequality and the fact that $L_1$ is near $L_2$ to obtain the inequality $\rho(v_1,v_2) \leq c\|L_1(v_1)-L_1(v_2)\| + K\rho(v_1,v_2)$, which implies \eqref{Cam bd}.
\end{proof}
We now introduce the renormalization function
\begin{align}\label{gamma}
    \gamma \coloneqq  \frac{\mathrm{tr}(A)}{A:A + b\cdot b}\in L^{\infty}(\Omega),
\end{align}
which is positive almost everywhere. The Cordes-type condition \eqref{Cordes} then implies
\begin{align}\label{gamma prop}
    (\gamma A - I_d):(\gamma A - I_d) + (\gamma b)\cdot (\gamma b) = d - \gamma \,\mathrm{tr}(A) \leq 1-\delta,
\end{align}
which allows us to show nearness of a renormalized version of the nondivergence-form differential operator
\begin{align}\label{L}
    L\colon H^2(\Omega)\cap H^1_0(\Omega) \rightarrow L^2(\Omega),\qquad v\mapsto A:D^2 v + b\cdot \nabla v
\end{align}
to the Laplacian as established in the following lemma.

\begin{lemma}[Nearness to the Laplacian]\label{Lmm: near Laplace}

Let $\Omega \subset \R^d$ be a bounded convex domain, $A\in L^{\infty}(\Omega;\R^{d\times d}_{\mathrm{sym}})$, $b\in L^{\infty}(\Omega;\R^{d})$, and suppose that \eqref{unifm ell} and \eqref{Cordes} hold. Recall the constants $\delta,\delta_0,\eta,C_P$ from \eqref{Cordes}, \eqref{d0 and eta}, and \eqref{Cp}. Then, introducing the operator
\begin{align}\label{L gamma}
    L_{\gamma}\colon H^2(\Omega)\cap H^1_0(\Omega) \rightarrow L^2(\Omega),\qquad v\mapsto \gamma A:D^2 v + \gamma b\cdot \nabla v
\end{align}
with $\gamma\in L^{\infty}(\Omega)$ given by \eqref{gamma}, we have that
\begin{align*}
  \forall v\in H^2(\Omega)\cap H^1_0(\Omega):\quad  \|\Delta v - L_{\gamma} v\|_{L^2}\leq \sqrt{1-\kappa}\, \|\Delta v\|_{L^2},
\end{align*}
where $\kappa \coloneqq (\delta - \delta_0)(1+\eta\, C_P^{-2}) \in (0,1]$. In particular,
\begin{align*}
    L_{\gamma}:H^2(\Omega)\cap H^1_0(\Omega) \rightarrow L^2(\Omega)\quad\text{is near}\quad \Delta:H^2(\Omega)\cap H^1_0(\Omega) \rightarrow L^2(\Omega)
\end{align*}
in the sense of Campanato.
\end{lemma}

\begin{proof}
Let $v\in H^2(\Omega)\cap H^1_0(\Omega)$. Noting that $\nabla v\in V$ and using \eqref{Cp}, we have that
\begin{align*}
C_\mathrm{P}\|\nabla v\|_{L^2}\leq \|D^2 v\|_{L^2} \leq \|\Delta v\|_{L^2},
\end{align*}
where the second inequality is the Miranda--Talenti (or Kadlec) estimate. Together with \eqref{gamma prop}, we then find that
\begin{align*}
\|\Delta v - L_{\gamma} v\|_{L^2}^2 &= 
    \|(\gamma A - I_d):D^2 v + \gamma b\cdot \nabla v\|_{L^2}^2 \\&\leq (1-\delta) \left(\|D^2 v\|_{L^2}^2 + \eta\, \|\nabla v\|_{L^2}^2\right) \\
    &\leq (1-\delta)\left(1+\eta\,C_\mathrm{P}^{-2}\right)\|\Delta v\|_{L^2}^2 = (1-\kappa)\|\Delta v\|_{L^2}^2,
\end{align*}
where $\kappa \coloneqq  (\delta - \delta_0)(1+\eta\, C_P^{-2})$. Note that $\kappa\in (0,1]$ since $\delta\in (\delta_0,1]$. 
\end{proof}

We are now in a position to prove well-posedness of the nondivergence-form problem \eqref{u prob} by combining \cref{Lmm: Campanato,Lmm: near Laplace}.

\begin{theorem}[Well-posedness of model problem]\label{Thm: Wp of u prob}
Let $\Omega \subset \R^d$ be a bounded convex domain, and $f\in L^2(\Omega)$. Let $A\in L^{\infty}(\Omega;\R^{d\times d}_{\mathrm{sym}})$ and $b\in L^{\infty}(\Omega;\R^{d})$ be such that~\eqref{unifm ell} and \eqref{Cordes} hold. Then, there exists a unique solution $u\in H^2(\Omega)\cap H^1_0(\Omega)$ to \eqref{u prob}, and we have the bound 
\begin{align*}
 \|\Delta u\|_{L^2} \leq \frac{\|\gamma\|_{L^{\infty}}}{1-\sqrt{1-\kappa}}\, \| f\|_{L^2},  
\end{align*}
where $\gamma\in L^\infty(\Omega)$ is given by \eqref{gamma}, and $\kappa\in (0,1]$ is defined in \cref{Lmm: near Laplace}.

\end{theorem}
\begin{proof}

We know from \cref{Lmm: near Laplace} that the operator $L_{\gamma}:H^2(\Omega)\cap H^1_0(\Omega) \rightarrow L^2(\Omega)$ given by \eqref{L gamma} is near the bijective operator $\Delta:H^2(\Omega)\cap H^1_0(\Omega) \rightarrow L^2(\Omega)$ in the sense of Campanato (with $c = 1$ and $K = \sqrt{1-\kappa}$). Hence, by \cref{Lmm: Campanato}, we find that $L_{\gamma}:H^2(\Omega)\cap H^1_0(\Omega) \rightarrow L^2(\Omega)$ is bijective and we have the bound
\begin{align*}
 \forall v\in H^2(\Omega)\cap H^1_0(\Omega)\colon \quad   \|\Delta v\|_{L^2} \leq \frac{1}{1-\sqrt{1-\kappa}}\, \| L_{\gamma} v\|_{L^2}.
\end{align*}
Since $\gamma > 0$ almost everywhere in $\Omega$ and $L_{\gamma} = \gamma L$ with $L:H^2(\Omega)\cap H^1_0(\Omega) \rightarrow L^2(\Omega)$ given by \eqref{L}, we find that $L$ is bijective and we have the bound
\begin{align*}
  \forall v\in H^2(\Omega)\cap H^1_0(\Omega)\colon\quad  \|\Delta v\|_{L^2} \leq \frac{\|\gamma\|_{L^{\infty}}}{1-\sqrt{1-\kappa}}\, \| L v\|_{L^2},
\end{align*}
which yields the claimed result.
\end{proof}

\subsection{Symmetric Lax--Milgram problem for $\boldsymbol{\nabla u}$}\label{Sec: 2.2}

Our starting point for constructing a numerical homogenization method is the observation that, inspired by \cite{Gal17b,Gal19}, $\nabla u$ can be characterized as the unique solution of a symmetric Lax--Milgram problem in the space $V$ defined in \cref{eq:V}.

For a chosen stabilization parameter $\sigma > 0$, we introduce the symmetric bilinear form $\fraka\colon V\times V\rightarrow \R$ defined by
\begin{align}\label{bilin form a}
    \fraka(\psi,\varphi) \coloneqq  (A:D\psi + b\cdot \psi, A:D\varphi+ b\cdot \varphi)_{L^2} +  \sigma\,(\mathrm{rot}(\psi),\mathrm{rot}(\varphi))_{L^2}
\end{align}
for $\psi,\varphi\in V$, where we used the notation
\begin{align*}
\mathrm{rot}(\psi) \coloneqq 
\left\{
\begin{aligned}
    &\partial_2 \psi_1 - \partial_1 \psi_2, &&\text{if } d=2,\\
    &(\partial_2 \psi_3 - \partial_3 \psi_2, \;\partial_3 \psi_1 - \partial_1 \psi_3, \;\partial_1 \psi_2 - \partial_2 \psi_1), &&\text{if } d=3.
\end{aligned}
\right.
\end{align*}
Then, we have the following result.

\begin{lemma}[Characterization of $\nabla u$]\label{Lmm: char of grad u}

Let the setting be as in \cref{Thm: Wp of u prob} with $d\in \{2,3\}$, and let $\sigma>0$. Let $\fraka\colon V\times V\rightarrow \R$ be the bilinear form given by \eqref{bilin form a}, and let $u\in H^2(\Omega)\cap H^1_0(\Omega)$ denote the unique strong solution to problem  \eqref{u prob}. Then, the following assertions hold.
\begin{itemize}
    \item[(i)] Coercivity of $\fraka$: there exists a constant $\alpha > 0$ such that for any $\varphi\in V$ it holds $\fraka(\varphi,\varphi) \geq \alpha\|D\varphi\|_{L^2}^2$.
    \item[(ii)] Boundedness of $\fraka$: there exists a constant $\beta > 0$ such that for any $\psi,\varphi\in V$ it holds $\lvert \fraka(\psi,\varphi)\rvert \leq \beta\|D\psi\|_{L^2}\|D\varphi\|_{L^2}$.
    \item[(iii)] There exists a unique $z\in V$ such that
\begin{align}\label{z prob}
  \forall \varphi\in V:\quad  \fraka(z,\varphi) = (f,A:D\varphi + b\cdot \varphi)_{L^2}.
\end{align}
Moreover, it holds $z = \nabla u$.
\end{itemize}
\end{lemma}

\begin{proof}
    (i) Let $\gamma\in L^\infty(\Omega)$ be the function defined in \eqref{gamma}, and let $\kappa\in (0,1]$ be the constant defined in \cref{Lmm: near Laplace}. Noting that $\sigma \|\gamma\|_{L^{\infty}}^2 > 0$, we fix $\varepsilon > 0$ such that $$\frac{1-\kappa}{\varepsilon} + \frac{(1-\sqrt{1-\kappa})^2}{1+\varepsilon} \leq \sigma \|\gamma\|_{L^{\infty}}^2,$$
    and we define the constants
    \begin{align*}
        r \coloneqq \frac{1-\kappa}{\varepsilon \|\gamma\|_{L^{\infty}}^2} \geq 0,\qquad \alpha \coloneqq \frac{(1-\sqrt{1-\kappa})^2}{(1+\varepsilon)\|\gamma\|_{L^{\infty}}^2} > 0.
    \end{align*}
    For any $\varphi\in V$, writing $\tilde{L}\varphi \coloneqq A:D\varphi + b\cdot \varphi$ and noting that $\sigma \geq r+\alpha$, we have
    \begin{align*}
        &\fraka(\varphi,\varphi)         \geq \|\tilde{L}\varphi\|_{L^2}^2 + r \|\mathrm{rot}(\varphi)\|_{L^2}^2 + \alpha \|\mathrm{rot}(\varphi)\|_{L^2}^2 \\
        &\geq \frac{\left[\|\tilde{L}\varphi\|_{L^2} + \sqrt{\varepsilon r}\, \|\mathrm{rot}(\varphi)\|_{L^2}\right]^2}{1+\varepsilon} + \alpha \|\mathrm{rot}(\varphi)\|_{L^2}^2 \\
        &\geq \frac{\left[\|\gamma\|_{L^{\infty}}   \|\tilde{L}\varphi\|_{L^2} + \sqrt{1-\kappa}\left( \|D\varphi\|_{L^2} -  \|\nabla\cdot \varphi\|_{L^2}\right)\right]^2}{(1+\varepsilon)\|\gamma\|_{L^{\infty}}^2} + \alpha \|\mathrm{rot}(\varphi)\|_{L^2}^2 \\
        &\geq \frac{\left[\|\gamma \tilde{L}\varphi\|_{L^2} + \|\nabla \cdot \varphi - \gamma \tilde L \varphi\|_{L^2} - \sqrt{1-\kappa} \|\nabla\cdot \varphi\|_{L^2}\right]^2}{(1+\varepsilon)\|\gamma\|_{L^{\infty}}^2} + \alpha \|\mathrm{rot}(\varphi)\|_{L^2}^2 \\
        &\geq \alpha \|\nabla\cdot \varphi\|_{L^2}^2 + \alpha \|\mathrm{rot}(\varphi)\|_{L^2}^2 \\
        &\geq \alpha \|D\varphi\|_{L^2}^2,
    \end{align*}
    where we have used the Maxwell (or Gaffney) estimate
    \begin{align*}
      \|D\varphi\|_{L^2}^2 \leq \|\nabla\cdot \varphi\|_{L^2}^2 + \|\mathrm{rot}(\varphi)\|_{L^2}^2,
    \end{align*}
    and in the third from last step that, by \eqref{gamma prop} and \eqref{Cp}, 
    \begin{align*}
\|\gamma \tilde L \varphi - \nabla\cdot \varphi\|_{L^2}^2 &=
        \|(\gamma A -I_d):D\varphi + \gamma b\cdot \varphi\|_{L^2}^2 \\&\leq (1-\delta) \left(\|D \varphi\|_{L^2}^2 + \eta \|\varphi\|_{L^2}^2\right) \\
    &\leq (1-\delta)\left(1+ \eta C_\mathrm{P}^{-2}\right)\|D \varphi\|_{L^2}^2 = (1-\kappa)\|D \varphi\|_{L^2}^2,
    \end{align*}
    with $\eta$ given by \eqref{d0 and eta}.
    
    (ii) This follows directly from H\"{o}lder's inequality, Poincar\'e's inequality \eqref{Cp}, and the fact that $\|\mathrm{rot}(\varphi)\|_{L^2}\lesssim \|D\varphi\|_{L^2}$ for any $\varphi\in V$.
    
    (iii) Note that $\|\cdot\|_{V}\coloneqq  \|D(\cdot)\|_{L^2}$ defines a norm on $V$ due to the Poincar\'{e} inequality \eqref{Cp}. Therefore, in view of (i)-(ii), the Lax--Milgram theorem applies to yield the existence and uniqueness of a function $z\in V$ satisfying \eqref{z prob}. Finally, since for the unique strong solution $u\in H^2(\Omega)\cap H^1_0(\Omega)$ to \eqref{u prob} we have that $\nabla u\in V$ and $\fraka(\nabla u,\varphi) = (f,A:D\varphi + b\cdot \varphi)_{L^2}$ for any $\varphi\in V$ (as $A:D^2 u + b\cdot \nabla u= f$ and $\mathrm{rot}(\nabla u) = 0$ almost everywhere in $\Omega$), it follows that $z = \nabla u$ by the uniqueness of the solution to~\eqref{z prob} in~$V$.
\end{proof}

\subsection{Recovery of $\boldsymbol u$}

Once the solution $z = \nabla u$ to \eqref{z prob} is computed, the recovery of $u$ from $z$ can be achieved as follows.

\begin{lemma}[Recovery of $u$ from $z$]\label{Lmm: recovery}
Let the situation be as in Lemma \ref{Lmm: char of grad u}. The solution $u\in H^2(\Omega)\cap H^1_0(\Omega)$ to \eqref{u prob} is the unique solution to the problem
\begin{align}\label{Poisson}
\left\{
\begin{aligned}
    \Delta u &= \nabla\cdot z &&\text{in } \Omega,\\
    u &= 0 &&\text{on } \partial\Omega,
\end{aligned}
\right.
\end{align}
where $z\in V$ denotes the unique solution to \eqref{z prob}.    
\end{lemma}
\begin{proof}
The result follows directly from \cref{Lmm: char of grad u}~(iii) together with the uniqueness of strong solutions to problem \eqref{Poisson}.
\end{proof}

\section{Ideal numerical homogenization}\label{Sec: 3}

In this section, we introduce an ideal numerical homogenization method that achieves optimal-order approximations without pre-asymptotic effects under minimal structural assumptions on the coefficients. Let $\Omega\subset \R^d$ be a bounded convex polyhedral domain in dimension $d\in \{2,3\}$, and consider a shape-regular hierarchy of geometrically conformal meshes $\{\mathcal{T}_H\}_{H>0}$ (cf.~\cite[Def.~1.55 \& 1.107]{ErG04}). Each mesh~$\mathcal{T}_H$ is a finite partition of $\overline{\Omega}$ into closed elements (either simplices or quadrilaterals/hexahedra). The mesh parameter $H > 0$ is defined by $H \coloneqq \max_{T \in \mathcal{T}_H} \operatorname{diam}(T)$. We assume that the family of meshes is quasi-uniform, that is, $H$ can be bounded from above by a constant multiple of the minimal element diameter. We denote by~$\mathcal{F}_H$ the set of all relatively closed mesh facets (edges for $d = 2$, faces for $d = 3$). To each facet $F \in \mathcal{F}_H$, we associate a fixed unit normal vector $n_F$.

\subsection{Fine-scale space}
A key ingredient in the numerical homogenization method considered here is the fine-scale space, consisting of functions that average out on coarse scales. This space is introduced via the concept of \emph{quantities of interest}, which are associated with each facet $F \in \mathcal F_H$ and defined as
\begin{align}\label{defQOI}
    q_F \in V^{\ast},\qquad \langle q_F,\varphi\rangle \coloneqq  \int_{F} \varphi \cdot n_F \ds\quad\text{for } \varphi \in V.
\end{align}
The fine-scale space is then defined as the closed subspace given by the intersection of the kernels of these quantities of interest, that is,
\begin{align}
	\label{eq:defW}
    W \coloneqq  \bigcap_{F \in \mathcal F_H} \mathrm{ker}(q_F) \subset V.
\end{align}
The following lemma plays an essential role in the analysis of the method.

\begin{lemma}[Local Poincar\'{e}-type inequality on $W$]\label{lem:poincare}
There exists a constant $c_\mathrm{p}>0$ independent of $H$ such that, for all $T \in \mathcal T_H$, it holds that
    \begin{equation}\label{eq:locpoincare}
        \forall w \in W:\quad\|w\|_{L^2(T)} \leq c_\mathrm{p} H\|D w\|_{L^2(T)}.
    \end{equation}
\end{lemma}
\begin{proof}
	This result can be derived from \cite[Lem.~B.66]{ErG04} using a transformation to the reference element and the corresponding estimates in~\cite[Lem.~1.101]{ErG04}.    
\end{proof}

\subsection{Ideal method}
The problem-adapted approximation space of the ideal numerical homogenization method is defined as the orthogonal complement of \(W\) with respect to the energy inner product \(\fraka\), that is,
\begin{equation}
	\label{eq:Zms}
	\tilde{V}_H \coloneqq \big\{ v \in V \with \fraka(v, w) = 0 \text{ for all } w \in W \big\}.
\end{equation}
Since \(W\) has finite codimension in \(V\), the space \(\tilde{V}_H\) is finite-dimensional.  
The use of tildes emphasizes that these spaces are specifically adapted to the problem.  

The prototypical LOD method is defined as the Galerkin projection onto the problem-adapted approximation space \(\tilde{V}_H\); that is, it seeks \(\tilde{z}_H \in \tilde{V}_H\) such that
\begin{align}\label{eq:protmethod}
    \forall \tilde v_H \in \tilde{V}_H:\quad\fraka(\tilde{z}_H,\tilde{v}_H) = (f,A:D\tilde{v}_H + b\cdot \tilde{v}_H)_{L^2}.
\end{align}
As will be seen later, both in theory (see Remark \ref{Rk: nec corr}) and in numerical experiments, this method does not directly yield a convergence rate in terms of $\|D(z - \tilde{z}_H)\|_{L^2}$, where $z \in V$ denotes the unique solution to \eqref{z prob}. To address this, we introduce additional corrections that allow us to extract rates from the right-hand side. 
To this end, let $Y_H$ be a scalar-valued finite element space associated with the coarse mesh $\mathcal T_H$, and let $\Pi_H\colon H^1(\Omega) \to Y_H$ denote an interpolation operator satisfying, for some $p \in \mathbb{N}_0$, the following estimate
\begin{equation}\label{eq:approxf}
    \forall k\in \{0,1,\ldots,p+1\}:\quad \|v - \Pi_H v \|_{L^2} \leq C_\Pi H^{k} \lvert v\rvert_{H^{k}}\;\;\text{for all }v \in H^{k}(\Omega),
\end{equation}
with a constant $C_\Pi>0$ independent of $H$. Note that $Y_H$ could, in principle, be replaced with any other space possessing approximation properties as in~\eqref{eq:approxf}.

The correction operator $\mathcal{Q}\colon Y_H \to W$ is defined as follows: for any $y_H \in Y_H$, we define $\mathcal{Q}y_H\in W$ to be the unique element of $W$ satisfying
\begin{equation}\label{eq:addcorr}
    \forall w \in W\colon\quad\fraka(\mathcal{Q}y_H, w) = (y_H, A:Dw + b\cdot w)_{L^2},
\end{equation}
and the ideal approximation $\tilde{z}_H \in \tilde{V}_H$ satisfying \cref{eq:protmethod} is then updated in a post-processing step, i.e.,
\begin{equation}\label{eq:idealmod}
    \hat{z}_H \coloneqq \tilde{z}_H + \mathcal{Q}(\Pi_H f),
\end{equation}
assuming that $f\in H^1(\Omega)$.

Higher-order convergence rates, independent of the regularity of the coefficients and the solution, are established in the following theorem.

\begin{theorem}[Ideal method]\label{thm:convideal} 
Suppose that $f\in H^{p+1}(\Omega)$. Let $z \in V$ denote the unique solution to~\cref{z prob} and let $\hat z_H$ be the post-processed ideal multiscale approximation as defined in~\cref{eq:idealmod}. Then, it holds
\begin{align}
    \|D(z - \hat{z}_H)\|_{L^2} &\leq \alpha^{-1}\big(\|A\|_{L^{\infty}} + C_\mathrm{P}^{-1}\|b\|_{L^{\infty}}\big) C_\Pi H^{p+1} \lvert f\rvert_{H^{p+1}},\label{eq:idealH1error}\\
    \|z - \hat{z}_H\|_{L^2} & \leq c_\mathrm{P} H \|D(z - \hat{z}_H)\|_{L^2},\label{eq:idealL2error}
\end{align}
where $\alpha>0$ is the coercivity constant for $\fraka$ from Lemma \ref{Lmm: char of grad u}, $C_\mathrm{P},c_\mathrm{P}>0$ are the Poincar\'{e} constants from \eqref{Cp} and \eqref{eq:locpoincare}, and $C_{\Pi}>0$ is the interpolation constant from \eqref{eq:approxf}.
\end{theorem}
\begin{proof}
    We use the coercivity of $\fraka$ from \cref{Lmm: char of grad u}, the definition  of $\mathcal{Q}$ in~\cref{eq:addcorr}, and the observation that $z-\hat{z}_H = z-\tilde{z}_H-\mathcal Q(\Pi_H f) \in W$ to obtain
\begin{align*}
    \alpha\|D(z - \hat{z}_H)\|_{L^2}^2  &\leq \fraka(z - \tilde{z}_H - \mathcal{Q}(\Pi_H f), z - \hat{z}_H)\\& =
    (f- \Pi_H f ,A:D(z - \hat{z}_H) + b\cdot (z - \hat{z}_H))_{L^2} \\
    &\leq \left( \|A\|_{L^{\infty}}\|D(z - \hat{z}_H)\|_{L^2} + \|b\|_{L^{\infty}} \|z - \hat{z}_H\|_{L^2}\right)\|f- \Pi_H f\|_{L^2}.
\end{align*}
In view of the Poincar\'e inequality \eqref{Cp}, we find that
\begin{align*}
    \|D(z - \hat{z}_H)\|_{L^2}\leq \alpha^{-1}\left( \|A\|_{L^{\infty}}  + C_\mathrm{P}^{-1}\|b\|_{L^{\infty}} \right)\|f- \Pi_H f\|_{L^2}.
\end{align*}
The estimate \eqref{eq:idealH1error} then follows using the interpolation error bound \eqref{eq:approxf}. Finally, since $z-\hat{z}_H\in W$, the $L^2$-estimate \cref{eq:idealL2error} is a consequence of \eqref{eq:idealH1error} in view of \cref{lem:poincare}. 
\end{proof}

Note that the variant~\cref{eq:idealmod} does not require $f$ to be known. It just requires to compute a set of additional corrections that are based on a low-dimensional space in which~$f$ may be suitably approximated. This slightly increased complexity regarding number of corrections makes it possible to obtain rates for the ideal method of essentially any order (depending on the amount of additional corrections and the corresponding spaces). This approach has some similarities to the enriched multiscale construction in~\cite{KalKMW25} (to capture time-dependence) and the right-hand side corrections in~\cite{HelM17}. However, the latter requires $f$ to be known.

Let us briefly comment on the necessity of correcting the ideal approximation~$\tilde z_H$ as done in \eqref{eq:idealmod}.

\begin{remark}[Necessity of a correction]\label{Rk: nec corr}
Noting that $z - \tilde z_H \in W$, and recalling~\eqref{eq:defW} and \eqref{eq:Zms}, an analogous argument to \cite[Thm.~3.9]{Altmann2021} yields the bound
\begin{align*}
    \|D(z - \tilde z_H)\|_{L^2} &\leq \alpha^{-1}\inf_{\{c_F\}_{F}\subset \R} \sup_{w\in W\backslash\{0\}} \frac{\fraka(z - \tilde{z}_H, w) - \sum_{F\in \mathcal{F}_H} c_F \langle q_F,w\rangle}{\|Dw\|_{L^2}} \\ &= \alpha^{-1} \inf_{\{c_F\}\subset \R} \sup_{w\in W\backslash\{0\}} \frac{(f,A:Dw+b\cdot w)_{L^2} - \sum_{F\in \mathcal{F}_H} c_F \langle q_F,w\rangle}{\|Dw\|_{L^2}},
\end{align*}
which reveals no convergence rate without an additional correction.
\end{remark}

\subsection{Alternative characterization of $\boldsymbol{\tilde V_H}$}
To better understand the structure of the space \( \tilde{V}_H \), we characterize the \(\fraka \)-orthogonal projection \( \mathcal{R} \colon V \to \tilde{V}_H \) via a saddle point formulation. This formulation uses Lagrange multipliers defined on the mesh skeleton $\Sigma \coloneqq \bigcup_{F \in \mathcal F_H} F$ in the space
\begin{equation*}
M_H \coloneqq \{v \in L^2(\Sigma) \with v|_F = \mathrm{const}\;\text{ for all } F \in \mathcal F_H\}.
\end{equation*}
Specifically, for any \( v \in V \), the projection \( \mathcal{R}v \in V \) and the associated Lagrange multiplier \( \lambda \in M_H \) are defined as the unique solution pair to
\begin{subequations}
	\label{eq:defR0} 
	\begin{align}
		&\qquad \qquad  \qquad \fraka (\mathcal R v, w)& +&  &\frakb(w,\lambda) & &=\quad  &0&&\text{for all }  w \in V, \quad &&\label{eq:defR10}\\
		&\qquad \qquad  \qquad \frakb(\mathcal R v,\mu)                   &   &         &    & &=\quad &\frakb(v,\mu)&&\text{for all } \mu \in M_H, \quad\label{eq:defR20}&&
	\end{align}
\end{subequations}
where $\frakb$ denotes the bilinear form 
\begin{equation*}
 \frakb \colon V \times M_H \to \mathbb R,\qquad   \frakb(v, \mu) \coloneqq \sum_{F \in \mathcal F_H} \int_F (v\cdot n_F) \mu \ds.
\end{equation*}
Indeed, equation~\cref{eq:defR10} ensures that \( \mathcal{R}v \in \tilde{V}_H \), while equation \cref{eq:defR20} implies that \( \mathcal{R}v - v \in W \). The well-posedness of \eqref{eq:defR0} then follows from classical saddle point theory (cf.~\cite[Cor.~4.2.1]{BofBF13}) and, in particular, relies on the inf--sup condition
\begin{equation}  \label{infsupbglob} 
  \adjustlimits\inf_{\mu \in M_H\backslash\{0\}} \sup_{v \in V\backslash\{0\}} \frac{\frakb (v, \mu)}{\|
  Dv \|_{L^2(\Omega)} \| \mu \|_{L^2(\Sigma)}} \gtrsim  H^{1/2} > 0.
\end{equation}
This inf--sup condition can be proved using bubble functions $\{\rho_F\with F \in \mathcal F_H\}$ such that $\rho_F \in V$ with $\operatorname{supp}(\rho_F) \subset \omega_F$, where $\omega_F$ denotes the union of the elements sharing the facet $F$ (one for boundary facets and two for interior facets). These bubble functions can be chosen such that
\begin{equation}\label{eq:deltaprop}
  \forall F,F^\prime \in \mathcal F_H:\quad \langle  q_{F^\prime},\rho_F\rangle = \delta_{F^\prime F},
\end{equation}
and such that the following stability estimates are satisfied:
\begin{equation}\label{eq:stabbubble}
\|\rho_F\|_{L^2(\omega_F)} \lesssim H^{-d/2+1},\qquad \|D \rho_F\|_{L^2(\omega_F)} \lesssim  H^{-d/2}.
\end{equation}
To prove the inf--sup condition~\cref{infsupbglob}, we construct, for any given $\mu \in M_H$, a function $v \in V$ that possesses the same quantities of interest, by
\begin{align*}
	v \coloneqq  \sum_{F \in \mathcal F_H} \mu|_F \rho_{F},
\end{align*}
where $\rho_F$ are the bubble functions introduced above.
Noting that, by construction, we have $\frakb(v, \mu) = |\mu|^2$ with $|\mu|^2 \coloneqq \sum_{F \in \mathcal F_H} (\mu|_F)^2$, and employing the locality of the bubble functions and the stability estimates from~\cref{eq:stabbubble}, we obtain
\begin{align*}
	\|D v\|_{L^2}^2 &\lesssim \sum_{F \in \mathcal F_H} (\mu|_{F})^2 \|D \rho_F\|_{L^2(\omega_F)}^2 
	\lesssim H^{-d} |\mu|^2.
\end{align*}
Combining these estimates, and noting that $\|\mu\|_{L^2(\Sigma)}^2 \lesssim H^{d-1}\lvert \mu\rvert^2$, we deduce the inf--sup stability bound~\cref{infsupbglob}.

\section{Exponential decay}\label{sec:decay}

In this section, we prove exponential decay properties of the operator $\mathcal R:V\rightarrow \tilde V_H$ defined in \cref{eq:defR0}, as well as of related operators. These decay properties will then allow us to define a localized version of $\mathcal R$ in \cref{sec:localization} and to construct a practical numerical homogenization method with locally computable basis functions in \cref{sec:pracmethod}.
To quantify this decay behavior, we introduce the concept of patches. Given an \emph{oversampling parameter} $\ell \in \mathbb N$, the $\ell$-th order patch of a union of elements $S \subset \TH$ is defined recursively for $\ell \geq 2$ by $\Nb^\ell(S) \coloneqq \Nb^1(\Nb^{\ell-1}(S))$, where $\Nb^1(S) \coloneqq \Nb(S)$ denotes the set of mesh elements sharing at least one node with the elements in $S$. We set $\Nb^0(S) \coloneqq S$.

\begin{theorem}[Exponential decay]\label{thm:dec}
    Let $S\subset \TH$ be a union of mesh elements, and let $(\psi,\lambda)\in V\times M_H$ be the unique solution to the problem
    \begin{subequations}
		\label{pbpsi} 
		\begin{align}
			&\qquad \qquad \fraka (\psi, v)& +&  &\frakb(v,\lambda) & &=\quad  &f_S(v)&&\;\mathrm{for}\;\mathrm{all}\; v \in V,\qquad \qquad\quad &&\label{eq:psi1}\\
			&\qquad\qquad \frakb(\psi,\mu)                   &   &         &    & &=\quad  &g_S(\mu)&&\;\mathrm{for}\;\mathrm{all}\; \mu \in M_H,\qquad \qquad\quad\label{eq:psi3}&&
		\end{align}
	\end{subequations}
	where $f_S \in V^*$ and $g_S\in (M_H)^*$ are such that $f_S(v) = 0$ for all $v \in V$ with $ \operatorname{supp}(v) \subset \overline{\Omega \setminus S}$, and $g_S(\mu) = 0$ for all $\mu \in M_H$ with $ \operatorname{supp}(\mu) \subset \overline{\Omega \setminus S}$. 
Then, there exists a constant \( c> 0 \), independent of \( H \), \( \ell \), and \( S \), such that
\begin{equation*}
\|D\psi\|_{L^2(\Omega\setminus \mathsf N^\ell(S))} \leq  \exp(-c \ell)\,\|D\psi\|_{L^2}
\end{equation*}
for all $\ell \in \mathbb N$.
\end{theorem}
\begin{proof} 
The proof relies on cut-off techniques, which are widely used in the context of multiscale methods (see, e.g., \cite{Mlqvist2020,Altmann2021}) and are adapted here to the present setting. We fix an integer $\ell \ge 1$ and denote by $\eta \in W^{1,\infty}(\Omega)$ the first-order Lagrange finite element cut-off function, defined with respect to $\mathcal T_H$, characterized by
  \begin{equation}
      \eta  =0 \text{ in } \Nb^{\ell - 1} (S),\quad
      \eta  =1 \text{ in } \Omega \setminus \Nb^{\ell} (S),  \label{eq:eta}
  \end{equation}
  with the transition region $R \coloneqq \Nb^\ell(S)\setminus \Nb^{\ell-1}(S)$. It satisfies
  \begin{equation}
    \label{eq:boundeta} 
    \| \nabla\eta \|_{L^{\infty}} \lesssim  
    H^{-1}.
  \end{equation} 
  We have that
\begin{align*}
\|D\psi\|_{L^2(\Omega\setminus N^\ell(S))}^2 &\leq \|D(\eta\psi)\|_{L^2}^2 
\\ &\lesssim \fraka(\eta\psi,\eta\psi)\\
&= \fraka(\psi,\eta\psi) - \fraka((1-\eta)\psi,\eta\psi)\\
&= f_S(\eta\psi) - \frakb(\eta\psi,\lambda) - \fraka((1-\eta)\psi,\eta\psi)\\
&\eqqcolon \Xi_1 - \Xi_2 - \Xi_3.
\end{align*}
Note that \(\Xi_1 = 0\) since \(\operatorname{supp}(\eta\psi)\subset \overline{\Omega\setminus N^{\ell-1}(S)}\subset \overline{\Omega \setminus S}\). 

In order to bound the term $\Xi_2$, let us first recall the trace inequality
\begin{align}\label{tr ineq}
    \forall v\in H^1(T_F;\R^d):\quad \|v\|_{L^2(F)} \lesssim H^{-1/2}\|v\|_{L^2(T_F)} + H^{1/2}\|Dv\|_{L^2(T_F)},
\end{align}
where \(T_F\) denotes an element such that \(F\subset T_F\); see, e.g., \cite[Lem.~1.49]{di2011mathematical}.
To estimate \(\lambda\vert_F\), we obtain by testing~\cref{eq:psi1} with the bubble function $\rho_F$ for each face \(F\subset R\) (recall \eqref{eq:deltaprop}, \eqref{eq:stabbubble}, and note $\operatorname{supp}(\rho_F)\subset \omega_F\subset \overline{\Omega \setminus S}$) that
\begin{align*}
\left\lvert \lambda\vert_F\right\rvert = |\frakb(\rho_F,\lambda)| &= | \fraka(\psi,\rho_F)| \\
&\lesssim \left( \|D\rho_F\|_{L^2(\omega_F)} + \|\rho_F\|_{L^2(\omega_F)}\right)\left( \|D\psi\|_{L^2(\omega_F)} + \|\psi\|_{L^2(\omega_F)}\right) \\ &\lesssim  H^{-d/2} \|D \psi\|_{L^2(\omega_F)},
\end{align*}
using the local Poincar\'{e} inequality from Lemma \ref{lem:poincare} and the bound \eqref{eq:stabbubble} in the final step. Since $\langle q_F,\psi\rangle = 0$ for all faces $F$ outside $\operatorname{int}(S)$ by \cref{eq:psi3}, we have
\begin{align*}
    \Xi_2 &= \sum_{F \in \mathcal F_H \with F\subset R} \int_F (\eta\psi\cdot n_F)\, \lambda|_F \ds \lesssim \sum_{F \in \mathcal F_H\with F\subset R} \|D\psi\|_{L^2(\omega_F)}^2\lesssim \|D\psi\|_{L^2(R)}^2,
\end{align*}
where we have used \cref{lem:poincare} and the trace inequality \eqref{tr ineq}. 

Finally, we estimate the term $\Xi_3$. We have
\begin{align*}
\Xi_3 &\lesssim \left(\|D((1-\eta)\psi)\|_{L^2(R)}+\|(1-\eta)\psi\|_{L^2(R)}\right)\left(\|D(\eta\psi)\|_{L^2(R)}+\|\eta\psi\|_{L^2(R)}\right) \\
&\lesssim \|D\psi\|_{L^2(R)}^2,
\end{align*}
where we used the product rule, the bound~\eqref{eq:boundeta} for the cut-off function, and \cref{lem:poincare}.

Since \(R = (\Omega \setminus \Nb^{\ell-1}(S)) \setminus (\Omega \setminus \Nb^{\ell}(S))\), we conclude that for a constant $C>0$ independent of $H$ and $\ell$, it holds that
  \begin{align*}
    \| D\psi \|_{L^2(\Omega \setminus \Nb^{\ell} (S))}^2 \leq C \| D
    \psi \|_{L^2(R)}^2 = C\| D\psi \|_{L^2(\Omega \setminus \Nb^{\ell - 1} (S))}^2 -
    C \| D\psi \|_{L^2(\Omega \setminus \Nb^{\ell} (S))}^2,
  \end{align*} 
which, after rearranging the terms, leads to
  \begin{equation*} 
   \| D\psi \|_{L^2(\Omega \setminus \Nb^{\ell} (S))}^2 \leq \frac{C}{1 + C}  \,\| D\psi \|_{L^2(\Omega \setminus \Nb^{\ell - 1}
     (S))}^2.
   \end{equation*} 
   Iterating the argument gives the assertion with decay rate $c\coloneqq\frac 12 \log \big(\frac{1+C}{C}\big)$.
\end{proof}

\section{Localization}
\label{sec:localization}

To localize the operator $\mathcal R$, we follow the approach proposed in the recent generalized framework for high-order LOD methods \cite{Hauck2026}, which has also been applied in \cite{Hauck2025HOStokes,Hauck2025MixedDim}. This framework enables a stable basis construction, ensuring that the error of the localized method does not grow as the coarse mesh is refined, provided that the oversampling parameter, defined as the number of element layers forming the oversampling domains, remains fixed. Such undesirable effects have been observed, for example, in \cite{MalP14,Maier2021}. Notably, this stability is achieved without relying on intricate bubble function constructions in practical implementations of the method (cf. \cite{Hauck2022,Dong2023}).

The key idea of this localization approach is to expresses the operator $\mathcal R$ as
\begin{equation}\label{defK}
	\mathcal{R} = \mathcal{I}_H - \mathcal{K},
\end{equation}
where the operator \( \mathcal{K} \colon V \to V \) will be characterized below, and \( \mathcal{I}_H \colon V \to V_H \) is a quasi-interpolation operator onto the first-order \(V\)-conforming finite element space~\( V_H \) associated with the mesh \( \mathcal{T}_H \). 
The quasi-interpolation operator \(\mathcal{I}_H\) (which is not necessarily a projection) is assumed to satisfy classical local approximation and stability estimates. That is, for all elements $T \in \mathcal{T}_H$, we have
\begin{equation}
	\label{eq:propIH}
\forall v\in V:\quad	H^{-1}\|v - \mathcal{I}_H v\|_{L^2(T)} + \|D(\mathcal{I}_H v)\|_{L^2(T)} \lesssim \|Dv\|_{L^2(\mathsf{N}(T))}.
\end{equation}
Moreover, \(\mathcal{I}_H\) should depend on its argument only through its quantities of interest defined in \cref{defQOI}, i.e., \(\mathcal{I}_H w = 0\) for all $w \in W$. Operators with these properties can be readily constructed for the problem at hand. A detailed description of such a construction is provided in \cref{sec:numexp}.

The operator \( \mathcal{K} \colon V \to V \), as introduced in \cref{defK}, is characterized for each \( v \in V \) via the unique solution \( (\mathcal{K}v, \lambda) \in V \times M_H \) to the saddle point problem
\begin{subequations}
	\label{eq:defK} 
	\begin{align}
		&\quad  \fraka (\mathcal K v, w)& +&  &\frakb(w,\lambda) & &=&&\  &\fraka(\mathcal I_H v,w)&&\text{for all } w \in V,\quad \quad &&\label{eq:defK1}\\
		&\quad \frakb(\mathcal K v,\mu)                   &   &         &    & &=&&\  &-\frakb(v-\mathcal I_H v,\mu)&&\text{for all } \mu \in M_H.\quad \quad\label{eq:defK2}&&
	\end{align}
\end{subequations}
We note that if $(\mathcal{K}v, \lambda)$ is the solution to \eqref{eq:defK}, then $((\mathcal{I}_H-\mathcal{K})v,-\lambda)$ is the solution to \eqref{eq:defR0}. The operator \( \mathcal{K} \) can be expressed as a sum of local element contributions, i.e.,
\begin{equation*}
	\mathcal{K} = \sum_{T \in \mathcal{T}_H} \mathcal{K}_T,
\end{equation*}
where, for each $T \in \TH$, the operator \( \mathcal{K}_T \colon V \to V \) is defined, for \( v \in V \), via the unique solution \( (\mathcal{K}_T v, \lambda_T) \in V \times M_H \) to the modified saddle point problem
	\begin{subequations}
	\label{eq:defKT} 
	\begin{align}
		& \fraka (\mathcal K_T v, w)& +&  &\frakb(w,\lambda_T) & &=   &&&\fraka_T(\mathcal I_H v,w)&&\text{for all } w \in V,\quad&&\label{eq:defKT1}\\
		& \frakb(\mathcal K_T v,\mu)                   &   &         &    & &=&  &&-\frakb_T(v-\mathcal I_H v,\mu)&&\text{for all } \mu \in M_H. \label{eq:defKT2}&&
	\end{align}
\end{subequations}
Here, for $T\in \mathcal{T}_H$, The bilinear form $\frakb_T$ denotes a localized version of $\frakb$ defined as
\begin{equation*}
    		\frakb_{T}(v,\mu) \coloneqq  \sum_{F \in \mathcal F_H \with F \subset \partial T} \frac{1}{N(F)}\,\mu\vert_{F}\,\langle q_{F},v\rangle,
\end{equation*}
where $N(F) \coloneqq \#\{ K \in \mathcal{T}_H \with F \subset \partial K\}$ is the number of elements sharing the face \( F \in \mathcal{F}_H \). The localized version $\fraka_T$ of the bilinear form $\fraka$ is defined by restricting the domain of integration in \cref{bilin form a} to the element $T$.

Due to the general setting considered in \cref{thm:dec}, the result applies to the function \( \mathcal{K}_T v \) for any \( v \in V \), and shows that it decays exponentially away from the element \( T \). This motivates the localization of the operator \( \mathcal{K}_T \) to $\ell$-th order patches around $T$. To this end, we introduce the corresponding localized spaces
\begin{align}\label{eq:locspaces}
\begin{split}
    V_T^\ell &\coloneqq \bigl\{ v \in V \with 
	v = 0 \;\text{a.e. in } \Omega \setminus \mathsf N^\ell(T) \bigr\}, \\
	M_T^\ell &\coloneqq \bigl\{ \mu \in M_H \with 
	\mu|_F = 0 \;\; \text{for all } F \in \mathcal F_H \setminus \mathcal F_T^\ell \bigr\},
\end{split}
\end{align}
where $\mathcal F_T^\ell \subset \mathcal F_H$ denotes the localized set of facets, defined by
\[
\mathcal F_T^\ell \coloneqq \bigl\{ F \in \mathcal F_H \with 
F \subset \mathsf N^\ell(T)\text{ and }
F \setminus \partial \mathsf N^\ell(T) \neq \emptyset \bigr\}.
\]
The localized operator \( \mathcal{K}_T^\ell \colon V \to V_T^\ell \) can then be defined, for any \( v \in V \), as the unique solution \( (\mathcal{K}_T^\ell v, \lambda_T^\ell) \in V_T^\ell \times M_T^\ell \) to the local saddle point problem
	\begin{subequations}
	\label{eq:defKTell} 
	\begin{align}
		& \fraka (\mathcal K_T^\ell v, w)& +&  &\frakb(w,\lambda_T^\ell) & &=&&  &\fraka_T(\mathcal I_H v,w)&&\text{for all } w \in V_T^\ell,\quad\quad &&\label{eq:defKTell1}\\
		& \frakb(\mathcal K_T^\ell v,\mu)                   &   &         &    & &=&&  &-\frakb_T(v-\mathcal I_H v,\mu)&&\text{for all } \mu \in M_T^\ell.\quad\label{eq:defKTell2}&&
	\end{align}
\end{subequations}
Finally, a localized version of the operator $\mathcal R$ can be defined by
\begin{equation}
	\label{eq:defRl}
	\mathcal R^\ell \coloneqq \mathcal I_H - \mathcal K^\ell,\qquad \mathcal K^\ell  \coloneqq \sum_{T \in \mathcal{T}_H} \mathcal K_T^\ell.
\end{equation}
The following theorem shows that the localized operator \( \mathcal{R}^\ell \) approximates \( \mathcal{R} \) exponentially well in the operator norm. Notably, it avoids the \( H^{-1} \) pre-factor typically arising in naive localization strategies (see, e.g.,~\cite{MalP14,Maier2021}) while also avoiding technical bubble constructions as in~\cite{Dong2023} in the implementation.

\begin{theorem}[Localization error for $\mathcal R$]
	\label{thm:locerr}
        For all \( v \in V \) and \( \ell \in \mathbb{N} \), we have 
		\begin{equation*}
			\|D(\mathcal{R}v - \mathcal{R}^\ell v)\|_{L^2} \lesssim \ell^{(d-1)/2} \exp(-c \ell)\|D v\|_{L^2},
		\end{equation*}
		where \( c > 0 \) is the constant from \cref{thm:dec}.
\end{theorem}
\begin{proof} 
The proof, which is an adaptation of that of~\cite[Thm.~5.1]{Hauck2026} to the present setting, is included for completeness. We abbreviate the localization error by \( e \coloneqq (\mathcal{R} - \mathcal{R}^\ell)v \) and note that \( e \in W \). Indeed, by \cref{defK,eq:defKT2,eq:defKTell2,eq:defRl}, we obtain $\mathfrak{b}(e,\mu)=0$ for all $\mu \in M_H$. Here we used that $b(\mathcal K_T^\ell v,\mu) = b(\mathcal K_T^\ell v,\mu|_{\mathcal F_T^\ell})$ for all $\mu \in M_H$, which follows from the definitions of $V_T^\ell$ and $M_T^\ell$ in~\cref{eq:locspaces}. We interpret $\mu|_{\mathcal F_T^\ell}$ as the restriction to the union of all facets in $\mathcal F_T^\ell$.
Thus,  
\begin{equation}
\begin{aligned}\label{eq:spliterror}
	\alpha\|De\|_{L^2}^2 
    \leq \fraka(e,e) = -\fraka(\mathcal R^\ell v,e) 
    =  \sum_{T \in \mathcal{T}_H} \Big(-\fraka_T(\mathcal I_Hv,e) + \fraka(\mathcal K_T^\ell v,e)\Big).
\end{aligned}
\end{equation}
In the following, we estimate each term in the summation on the right hand side separately.
Let $\eta_T\in W^{1,\infty}(\Omega)$ denote the cut-off function from \cref{eq:eta} with $S=T$. Noting that \( \eta_T =0 \) on \( T \), and using \cref{eq:defKTell1} with the test function \(w =  (1 - \eta_T)e \in V_T^\ell \), we find
\begin{align*}
	-\fraka_T(\mathcal I_Hv,e)+\fraka(\mathcal K_T^\ell v,e) &= -\fraka_T(\mathcal I_Hv,(1-\eta_T)e)+\fraka(\mathcal K_T^\ell v,(1-\eta_T)e + \eta_T e)\\
	&= - \frakb((1-\eta_T) e,\lambda_T^\ell) +\fraka(\mathcal K_T^\ell v,\eta_T e)\eqqcolon \Theta_1+\Theta_2.
\end{align*}
To estimate the term $\Theta_1$, we decompose~\( \lambda^\ell_T = \lambda^\mathrm{in}_T + \lambda^\mathrm{out}_T \), where 
\begin{equation*} 
    \lambda^\mathrm{in}_T = \lambda_T^\ell\vert_{\mathcal F_T^{\ell-1}},\qquad \lambda^\mathrm{out}_T = \lambda^\ell_T|_{\mathcal F_H \setminus \mathcal F_T^{\ell-1}}, 
\end{equation*}
both extended by zero to all other facets so that $\lambda^\mathrm{in}_T,\lambda^\mathrm{out}_T\in M_H$. Since \( \eta_T  = 0 \) in  \( \mathsf{N}^{\ell-1}(T) \), we have that \( \frakb((1-\eta_T)e, \lambda_T^\mathrm{in}) = \frakb(e, \lambda_T^\mathrm{in}) = 0 \) by \eqref{eq:defKT2} and~\eqref{eq:defKTell2}. Thus, with $R_T \coloneqq \operatorname{int}\big( \Nb^{\ell}(T) \setminus \Nb^{\ell - 1}(T) \big)$, we obtain with the same arguments as for $\Xi_2$ in the proof of \cref{thm:dec} that
\begin{align*}
	\Theta_1 &= -\frakb((1-\eta_T)e,\lambda_T^\mathrm{out}) \\&= -\sum_{F \in \mathcal F_H \with F\subset R_T} \int_F (1-\eta_T)e\cdot n_F \lambda_T^\mathrm{out}\ds
    \\&
    \lesssim \|D (\mathcal K_T^\ell v)\|_{L^2(R_T)}\|De\|_{L^2(R_T)}.
\end{align*}
Turning to $\Theta_2 = \fraka(\mathcal K_T^\ell v,\eta_T e)$, we remark that the element contributions to this term also vanish for all the mesh elements outside $R_T$. As above, we arrive at 
\begin{align*}
	\Theta_2 \lesssim  \|D(\mathcal K_T^\ell v)\|_{L^2(R_T)}\big(\|De\|_{L^2(R_T)}+ \|e\|_{L^2(R_T)}\big)
        \lesssim \|D(\mathcal K_T^\ell v)\|_{L^2(R_T)}\|De\|_{L^2(R_T)}.
\end{align*}
Combining the estimates for $\Theta_1$ and $\Theta_2$, we obtain
\begin{equation}\label{eq:bounde}
\begin{aligned}
	\|De\|_{L^2}^2 &\lesssim  \sum_{T \in \TH} \|D(\mathcal K_T^\ell v)\|_{L^2(R_T)}\|De\|_{L^2(R_T)} \\&\lesssim  \exp(-c\ell) \sum_{T \in \TH}\|D(\mathcal K_T^\ell v)\|_{L^2(\Nb^{\ell}(T))}\|De\|_{L^2(R_T)},
\end{aligned}
\end{equation}
where we have applied~\cref{thm:dec} to \cref{eq:defKTell}, treating the patch $\mathsf N^\ell(T)$ as the whole domain.
To pass from the norm of $\mathcal K_T^\ell v$ to that of $v$, we again use classical saddle point theory (see, e.g.,~\cite[Cor.~4.2.1]{BofBF13}), 
noting that the inf--sup constant of $\frakb$ for the pair $(V_T^\ell,M_T^\ell)$ is of order $H^{1/2}$ following the arguments in \cref{infsupbglob}. Hence, 
\begin{equation*}
\begin{aligned}
   \| D(\mathcal{K}_T^{\ell} v)\|_{L^2(\Nb^{\ell} (T))} & \lesssim
  \sup_{w \in V_T^\ell\backslash\{0\}}\frac{\fraka_T (\mathcal{I}_H v, w)}{\| Dw
  \|_{L^2}} + H^{-1/2} \sup_{\mu \in M_T^\ell\backslash\{0\}} \frac{\frakb_T (v - \mathcal I_Hv,
  \mu)}{\| \mu \|_{L^2(\Sigma)}}\\
  & \lesssim  \| v\|_{H^1(\Nb (T))}
\end{aligned}
\end{equation*}
where, in the last step, to estimate the first term involving $\mathfrak{a}_T$, we used the $L^\infty$-bounds on the coefficients $A$ and $b$, the properties of $\mathcal I_H$ as stated in \cref{eq:propIH}, and the Poincaré inequality for $w$, cf.~\cref{Cp}. The second term involving $\mathfrak{b}_T$ is treated analogously, additionally invoking the classical trace inequality from~\cref{tr ineq}. 

Returning to~\eqref{eq:bounde}, we conclude that
\begin{align*}
    \|De\|_{L^2}^2&\lesssim  \exp(-c\ell) \sqrt{\sum_{T \in \TH} \|v\|_{H^1(\Nb(T))}^2}\sqrt{\sum_{T \in \TH} \|De\|_{L^2(R_T)}^2} \\
    & \lesssim  \ell^{(d-1)/2}\exp(-c\ell) \|Dv\|_{L^2} \|De\|_{L^2}, 
\end{align*}
where we have used the fact that each element $K\in\TH$ belongs to at most $\mathcal O(\ell^{d-1})$ rings $R_T$ for different $T\in\TH$, as well as the Poincaré inequality for $v$, cf.~\cref{Cp}. Dividing by $\|De\|_{L^2}$ yields  the assertion.
\end{proof}
The remainder of this section is devoted to the localization of the operator $\mathcal Q\colon Y_H \to W$ defined in~\eqref{eq:addcorr}. As above, for a given $y_H \in Y_H$, the function $\mathcal Q y_H$ is characterized by the solution pair $(\mathcal Q y_H,\zeta) \in V \times M_H$ satisfying
\begin{subequations} 
	\begin{align*}
		&\fraka (\mathcal Q y_H, w) \hspace*{-1ex} & + \hspace*{-1ex} &  &\frakb(w,\zeta) \hspace*{-1ex} & &= \;\; &(y_H,A : Dw + b\cdot w)_{L^2}&&\text{for all }  w \in V, &&\\
		&\frakb(\mathcal Q y_H,\mu)                   &   &         &    & &= \;\;&0&&\text{for all }  \mu \in M_H. &&
	\end{align*}
\end{subequations}
For any $T \in \mathcal T_H$, we define the corresponding element contributions $(\mathcal Q_T y_H,\zeta_T) \in V \times M_H$ as the solution to
\begin{subequations}
	\begin{align*}
		&\fraka (\mathcal Q_T y_H, w) \hspace*{-1.5ex} & + \hspace*{-1.5ex} &  &\frakb(w,\zeta_T) \hspace*{-1ex} & &= \;\; &(y_H,A : Dw + b\cdot w)_{L^2(T)}&&\hspace{-0.17cm}\text{for all } w \in V, &&\\
		&\frakb(\mathcal Q_T y_H,\mu)                   &   &         &    & &= \;\;&0&&\hspace{-0.17cm}\text{for all } \mu \in M_H.&&
	\end{align*}
\end{subequations}
Note that $\mathcal Q = \sum_{T\in \mathcal T_H} \mathcal Q_T$. Moreover, the operators $\mathcal Q_T$ exhibit exponential decay away from $T$, as a direct consequence of~\cref{thm:dec}. Their localized counterparts are obtained from the solution pair $(\mathcal Q_T^\ell y_H,\zeta_T^\ell) \in V_T^\ell \times M_T^\ell$ to
\begin{subequations}
	\label{eq:defQTell} 
	\begin{align}
		&\fraka (\mathcal Q_T^\ell y_H, w) \hspace*{-1.5ex} & + \hspace*{-1.5ex} &  &\frakb(w,\zeta_T^\ell) \hspace*{-1ex} & &= \;\; &(y_H,A : Dw + b\cdot w)_{L^2(T)}&&\hspace{-0.17cm}\text{for all } w \in V_T^\ell, &&\label{eq:defR13}\\
		&\frakb(\mathcal Q_T^\ell y_H,\mu)                   &   &         &    & &= \;\;&0&&\hspace{-0.17cm}\text{for all } \mu \in M_T^\ell. \label{eq:defR23}&&
	\end{align}
\end{subequations}
For the localized version $\mathcal Q^\ell \coloneqq \sum_{T \in \mathcal T_H} \mathcal Q_T^\ell$ of the operator $\mathcal Q$, the following exponential approximation result holds.
\begin{theorem}[Localization error for $\mathcal Q$]
	\label{thm:locerrQ}
        For all \( y_H \in Y_H \) and \( \ell \in \mathbb{N} \), we have 
		\begin{equation*}
			\|D(\mathcal{Q}y_H - \mathcal{Q}^\ell y_H)\|_{L^2} \lesssim \ell^{(d-1)/2} \exp(-c \ell)\|y_H\|_{L^2},
		\end{equation*}
		where \( c > 0 \) is the constant from \cref{thm:dec}.
\end{theorem}
\begin{proof} 
    The proof is very similar to the proof of \cref{thm:locerr} and is therefore reduced to the essential steps. The missing precise arguments are exactly as in the above proof.  Let $e \coloneqq (\mathcal Q - \mathcal Q^\ell) y_H$. With the cut-off function from \cref{eq:eta} with $S=T$, now denoted by \( \eta_T \), we obtain
    \begin{align*}
        \alpha \|De\|_{L^2}^2 &\leq \fraka((\mathcal Q - \mathcal Q^\ell)y_H,e) = (y_H,A:De + b\cdot e)_{L^2}-\fraka(\mathcal Q^\ell y_H,e)\\
        & = \sum_{T \in \mathcal T_H} \Big(-\fraka(\mathcal Q_T^\ell y_H,(1-\eta_T)e + \eta_Te) \\[-1.5ex]
        &\qquad\qquad\quad + (y_H,A:D[(1-\eta_T)e] + b\cdot (1-\eta_T) e)_{L^2(T)}\Big)\\
        & = \sum_{T \in \mathcal T_H} \frakb((1-\eta_T)e,\zeta_T^\ell\vert_{\Sigma\setminus\Nb^{\ell-1}(T)})  - \fraka(\mathcal Q_T^\ell y_H,\eta_Te)\\
        & \lesssim \sum_{T \in \TH} \|D(\mathcal Q_T^\ell y_H)\|_{L^2(R_T)}\|De\|_{L^2(R_T)} \\&\lesssim  \exp(-c\ell) \sum_{T \in \TH}\|D(\mathcal Q_T^\ell y_H)\|_{L^2(\Nb^{\ell}(T))}\|De\|_{L^2(R_T)},
    \end{align*}
    where we have applied~\cref{thm:dec}  to~\cref{eq:defQTell}, treating $\mathsf N^\ell(T)$ as the whole domain.  

    The result then follows as in \cref{thm:locerr}, using a discrete Cauchy--Schwarz inequality and the stability of $\mathcal Q_T^\ell$. The latter holds since
    \begin{align*}
    \alpha \|D (\mathcal Q_T^\ell y_H)\|_{L^2(\mathsf N^\ell(T))}^2 &\leq 
        \fraka(\mathcal Q_T^\ell y_H, \mathcal Q_T^\ell y_H)\\ &= (y_H, A : D (\mathcal Q_T^\ell y_H) + b \cdot (\mathcal Q_T^\ell y_H))_{L^2(T)}\\& 
        \leq \|y_H\|_{L^2(T)} (\|A\|_{L^{\infty}} + C_\mathrm{P}^{-1} \|b\|_{L^{\infty}}) \|D (\mathcal Q_T^\ell y_H)\|_{L^2(\mathsf N^\ell(T))},
    \end{align*}
    which concludes the proof. 
\end{proof}

\section{Practical multiscale method}\label{sec:pracmethod}

In this section, we introduce a practical multiscale method based on locally computed basis functions. This is justified by the exponential decay of the globally defined prototypical basis functions and the localization results in \cref{thm:locerr} and \cref{thm:locerrQ}. For each $F\in \mathcal{F}_H$, we define the localized basis function
\begin{equation}
	\label{eq:locbasis}
	\tilde \varphi_{F}^\ell \coloneqq \mathcal{R}^\ell \rho_{F}.
\end{equation}
We define the localized multiscale space as their span, i.e., 
\begin{equation*}
	\tilde V_H^\ell \coloneqq \operatorname{span}\{\tilde \varphi_{F}^\ell\with F \in \mathcal{F}_H\}.
\end{equation*}
We emphasize that the operator \( \mathcal{R}^\ell \) depends on its argument only through the quantities of interest, as introduced in \cref{defQOI}. Consequently, the definition~\cref{eq:locbasis} of~\( \tilde\varphi^\ell_{F} \) is independent of the specific choice of bubble functions, provided the Kronecker-delta condition~\cref{eq:deltaprop} on the quantities of interest is satisfied. 

The localized multiscale method now seeks the unique function $\tilde{z}_H\in \tilde V_H^\ell$ such that
\begin{align}\label{eq:locmethod}
	\forall \tilde v_H^\ell \in \tilde V_H^\ell:\quad\fraka(\tilde{z}_H^\ell,\tilde v_H^\ell)=(f,A : D\tilde v_H^\ell + b\cdot \tilde v_H^\ell)_{L^2}.
\end{align}
As a post-processing step, we then define
\begin{equation}\label{eq:postprocessed}
    \hat{z}_H^\ell \coloneqq \tilde{z}_H^\ell + \mathcal Q^\ell (\Pi_H f).
\end{equation}
The following theorem provides convergence results for the presented method.
\begin{theorem}[Localized  method]\label{thm:convergenceloc} Let $z \in V$ be the solution to~\cref{z prob} and let $\hat z_H^\ell$ be the post-processed multiscale approximation defined in \cref{eq:postprocessed}. Then, we have the error bounds
	\begin{align}
		\| D(z - \hat{z}_H^{\ell})\|_{L^2} & \lesssim  H^{p+1} \lvert f\rvert_{H^{p+1}} +  \ell^{(d-1)/ 2} \exp (- c \ell) \|f\|_{L^2},  \label{eq:errestpracH1}\\
\|z - \hat{z}_H^{\ell} \|_{L^2} & \lesssim H^{p+2}|f|_{H^{p+1}}+ \ell^{(d-1)/ 2} \exp (- c \ell) \|f\|_{L^2}.  \label{eq:errestpracL2}
	\end{align}
\end{theorem}
\begin{proof}
    Let $\hat z_H$ be the solution to~\cref{eq:idealmod} and $\tilde z_H\in \tilde{V}_H$ the solution to~\cref{eq:protmethod}. By the triangle inequality, we have
    \begin{equation}\label{eq:spliterr}
        \| D( z - \hat z_H^\ell)\|_{L^2} \leq \| D(z - \hat z_H)\|_{L^2} + \| D(\mathcal Q(\Pi_H f) - \mathcal Q^\ell(\Pi_H f))\|_{L^2} + \|D(\tilde z_H - \tilde z_H^\ell)\|_{L^2}.
    \end{equation}
    The first two terms can be bounded with \cref{thm:convideal} and \cref{thm:locerrQ}, respectively. For the third term, we observe with $\tilde z_H = \mathcal R \tilde z_H$ that
    \begin{equation}\label{eq:spliterr2}
        \| D(\tilde z_H - \tilde z_H^\ell) \|_{L^2} \leq \| D( \mathcal R \tilde z_H - \mathcal R^\ell \tilde z_H )\|_{L^2} + \|D(\mathcal R^\ell \tilde z_H - \tilde z_H^\ell)\|_{L^2}. 
    \end{equation}
    The first term can be estimated with Theorem~\ref{thm:locerr}. For the second term in~\eqref{eq:spliterr2}, using that $v = \mathcal R^\ell v$ for $v \in \tilde V_H^\ell$ and writing $\tilde L \varphi \coloneqq A:D\varphi + b\cdot \varphi$ for $\varphi\in V$, we obtain
    \begin{align*}
        \| D(&\mathcal R^\ell \tilde z_H - \tilde z_H^\ell) \|_{L^2}\\ &\lesssim \sup_{v \in \tilde V_H^\ell \setminus \{0\}} \frac{\fraka(\mathcal R^\ell \tilde z_H - \tilde z_H^\ell, v)}{\|D v\|_{L^2}} \\
        & = \sup_{v \in \tilde V_H^\ell \setminus \{0\}} \frac{\fraka(\mathcal R^\ell \tilde z_H - \mathcal R \tilde z_H + \tilde z_H - \tilde z_H^\ell, v)}{\|D v\|_{L^2}}\\
        & \lesssim \| D( \mathcal R \tilde z_H - \mathcal R^\ell \tilde z_H)\|_{L^2} +  \sup_{v \in \tilde V_H^\ell \setminus \{0\}} \frac{\fraka(\tilde z_H, v) - (f, \tilde L v)_{L^2}}{\|D v\|_{L^2}}\\
        &= \| D( \mathcal R \tilde z_H - \mathcal R^\ell \tilde z_H)\|_{L^2}  +  \sup_{v \in \tilde V_H^\ell \setminus \{0\}} \frac{\fraka(\tilde z_H, \mathcal R^\ell v - \mathcal R v) - (f, \tilde L (\mathcal R^\ell v - \mathcal R v))_{L^2}}{\|D v\|_{L^2}}\\
        &\lesssim \ell^{(d-1)/ 2} \exp (- c \ell) \|f\|_{L^2},
    \end{align*}
    where we have used \cref{thm:locerr} and the stability of $\tilde z_H$ in the final step. 
    
    Now, we combine this with~\eqref{eq:spliterr2} and once again \cref{thm:locerr}, which yields
    \begin{equation}\label{eq:boundidealloc}
        \| D(\tilde z_H - \tilde z_H^\ell) \|_{L^2} \lesssim \ell^{(d-1)/ 2} \exp (- c \ell) \|f\|_{L^2}.
    \end{equation}
    Finally, we plug this into~\eqref{eq:spliterr} and use \cref{thm:convideal} and \cref{thm:locerrQ} to obtain
    \begin{equation*}
        \| D( z - \hat z_H^\ell)\|_{L^2} \lesssim H^{p+1} \lvert f\rvert_{H^{p+1}} + \ell^{(d-1)/ 2} \exp (- c \ell) \|f\|_{L^2},
    \end{equation*}
    which is the claimed bound \eqref{eq:errestpracH1}.

    To show the $L^2$-estimate, we define $r \in V$ as the auxiliary solution to
    \begin{equation}\label{eq:auxsol}
        \forall v\in V\colon \quad\fraka(r,v) = (z - \hat z_H^\ell,v)_{L^2}.
    \end{equation}
    Further, $\tilde r_H \in \tilde V_H$ is the corresponding ideal multiscale approximation that solves
    \begin{equation*}
        \forall \tilde v_H\in \tilde V_H\colon \quad\fraka(\tilde r_H,\tilde v_H) = (z - \hat z_H^\ell,\tilde v_H)_{L^2}.
    \end{equation*}
    Note that $r - \tilde r_H \in W$ by the Galerkin orthogonality. Therefore, we have with \cref{lem:poincare} that
    \begin{equation*}
        \|D(r - \tilde r_H)\|_{L^2}^2 \lesssim \fraka(r - \tilde r_H, r - \tilde r_H) = (z - \hat z_H^\ell,r - \tilde r_H)_{L^2} \lesssim H\| z - \hat z_H^\ell\|_{L^2} \|D(r - \tilde r_H)\|_{L^2}  
    \end{equation*}
    and thus, 
    \begin{equation}\label{eq:errr}
        \|D(r - \tilde r_H)\|_{L^2} \lesssim H \| z - \hat z_H^\ell\|_{L^2}.
    \end{equation}
    Using $v = z - \hat z_H^\ell$ as a test function in~\eqref{eq:auxsol}, we obtain with~\cref{eq:errr}, \cref{thm:convideal}, and \cref{thm:locerrQ} that
    \begin{align*}
        \|z - \hat z_H^\ell\|_{L^2}^2 & = \fraka(r,z - \hat z_H^\ell) \\&= \fraka(r,z - \hat z_H) + \fraka(r,\mathcal Q (\Pi_H f) - \mathcal Q^\ell (\Pi_Hf)) + \fraka(r,\tilde z_H - \tilde z_H^\ell) \\& = \fraka(r - \tilde r_H,z - \hat z_H) + \fraka(r-\tilde r_H,\mathcal Q (\Pi_H f) - \mathcal Q^\ell (\Pi_Hf)) + \fraka(r,\tilde z_H - \tilde z_H^\ell)\\&\lesssim  H \| z - \hat z_H^\ell\|_{L^2} \big(\| D(z - \hat z_H)\|_{L^2} + \ell^{(d-1)/ 2} \exp (- c \ell) \|f\|_{L^2}\big)    \\&\qquad+ \| z - \hat z_H^\ell\|_{L^2}\| \tilde z_H - \tilde z_H^\ell\|_{L^2}.
    \end{align*}
   Using the Poincar\'e-type inequality~\cref{Cp} and employing~\cref{eq:boundidealloc}, we finally get
    \begin{align*}
        \|z - \hat z_H^\ell\|_{L^2}^2 &\lesssim \| z - \hat z_H^\ell\|_{L^2} \big(H \| D(z - \hat z_H)\|_{L^2} + H\ell^{(d-1)/ 2} \exp (- c \ell) \|f\|_{L^2} \\&\qquad+ \ell^{(d-1)/ 2} \exp (- c \ell) \|f\|_{L^2}\big).
    \end{align*}
    In view of \eqref{eq:idealH1error}, we obtain the claimed bound \eqref{eq:errestpracL2}.
\end{proof}

The practical multiscale method analyzed in Theorem \ref{thm:convergenceloc} produces an approximation $\hat z_H^\ell$ to $z = \nabla u$ in the $H^1$-norm, where we recall that $u\in H^2(\Omega)\cap H^1_0(\Omega)$ is the unique strong solution to \eqref{u prob}. Let us conclude this section with a brief discussion on the recovery of an approximation to $u$ in the $L^2$-norm.

In the spirit of Lemma \ref{Lmm: recovery}, let $\hat u\in H^2(\Omega)\cap H^1_0(\Omega)$ denote the unique solution to the problem
\begin{align}\label{uhat}
\left\{
\begin{aligned}
\Delta \hat u &= \nabla\cdot \hat z_H^\ell\quad\text{in }\Omega,\\
    \hat u &= 0 \qquad\quad\text{ on }\partial\Omega,
\end{aligned}
\right.
\end{align}
and let $\hat u_{\mathcal{H}}$ denote an $H^1_0(\Omega)$-conforming $\mathcal{P}^1$-finite element approximation to $\hat u$ on a (coarse) mesh $\mathcal{T}_{\mathcal{H}}$. Here, $\mathcal{T}_{\mathcal{H}}$ is an element of a quasi-uniform and shape-regular family $\{\mathcal{T}_{\mathcal{H}}\}_{\mathcal{H} > 0}$ of geometrically conformal meshes of $\Omega$, where $\mathcal{H}\in (0,H]$ denotes the maximum diameter of elements in $\mathcal{T}_{\mathcal{H}}$. Note that $\mathcal{T}_{\mathcal{H}} = \mathcal{T}_{H}$ is an admissible choice. In view of Theorem \ref{Thm: Wp of u prob}, Lemma \ref{Lmm: recovery}, and \eqref{uhat}, we have that $\|u-\hat u \|_{H^1}\lesssim \|z-\hat z_H^\ell\|_{L^2}$ and
\begin{align*}
    \|\hat u-\hat u_{\mathcal{H}}\|_{L^2} + \mathcal{H}\, \lvert \hat u-\hat u_{\mathcal{H}} \rvert_{H^1}  \lesssim \mathcal{H}^2 \|\nabla\cdot \hat z_H^\ell\|_{L^2} \lesssim  \mathcal{H}^2\|D( z-\hat z_H^\ell)\|_{L^2} +  \mathcal{H}^2 \| f\|_{L^2}.
\end{align*}
Hence, by the triangle inequality and Theorem \ref{thm:convergenceloc}, we obtain the $L^2$-error bound
\begin{align*}
    \|u - \hat u_{\mathcal{H}}\|_{L^2} \lesssim H^{p+2}|f|_{H^{p+1}}+ \left(\ell^{(d-1)/ 2} \exp (- c \ell) + \mathcal{H}^2\right) \|f\|_{L^2}.
\end{align*}
In particular, if $\mathcal{T}_{\mathcal{H}}$ is constructed such that $\mathcal{H} \lesssim H^{1+\frac{p}{2}}$, then
\begin{align*}
    \|u - \hat u_{\mathcal{H}}\|_{L^2} \lesssim H^{p+2}\|f\|_{H^{p+1}}+ \ell^{(d-1)/ 2} \exp (- c \ell) \|f\|_{L^2}.
\end{align*}

\section{Implementation and numerical experiments}
\label{sec:numexp}

In this section, we discuss the implementation of the proposed multiscale method and present numerical experiments supporting the theoretical results of this paper. 
Note that all numerical experiments can be reproduced using the code available at \url{https://github.com/moimmahauck/NonDiv_LOD}.

\subsection{Choice of quasi-interpolation $\boldsymbol{\mathcal I_H}$}
For the definition of the operator $\mathcal K$ in~\cref{defK}, we require a quasi-interpolation operator $\mathcal I_H:V \to V_H$ satisfying the approximation and stability properties stated in \cref{eq:propIH}. 
In the two-dimensional case, a possible choice of $\mathcal I_H$ is obtained by prescribing its nodal values at each interior node $z$ via
\begin{equation}\label{eq:IHdef}
(\mathcal I_H v)(z)
\coloneqq  
\begin{bmatrix}
n_{F_1}^1 & n_{F_1}^2\\
n_{F_2}^1 & n_{F_2}^2
\end{bmatrix}^{-1}
\begin{bmatrix}
|F_1|^{-1}\int_{F_1} v \cdot n_{F_1}\ds \\
|F_2|^{-1}\int_{F_2} v \cdot n_{F_2}\ds
\end{bmatrix},
\end{equation}
where $F_1$ and $F_2$ are two faces adjacent to $z$ such that their corresponding unit normal vectors $n_{F_1} = (n_{F_1}^1,n_{F_1}^2)$ and $n_{F_2}= (n_{F_2}^1,n_{F_2}^2)$ are linearly independent. For boundary nodes, an analogous construction is applied, yielding a discrete function~$v_H$ that does not, in general, satisfy the tangential boundary conditions of the space $V$. Conformity with respect to the space $V$ is then enforced by a suitable post-processing step: namely, the nodal values of $v_H$ at the boundary are modified so as to eliminate its tangential component $v_H - (v_H\cdot n)n$. In the simple case of a square domain, this procedure reduces to setting one component of the nodal boundary values of~$v_H$ to zero.
Construction \cref{eq:IHdef}, together with the post-processing step to enforce conformity with the space $V$, extends naturally to three dimensions by selecting, for each interior node $z$, three adjacent faces whose normal vectors are linearly independent. The stability and approximation properties stated in \cref{eq:propIH} can be verified by following the arguments in \cite[Ch.~1.6]{ErG04}.

\subsection{Fine-scale discretization}
To compute the basis functions of the proposed multiscale method, the problems 
\cref{eq:defKTell,eq:defQTell}, associated with the operators $\mathcal{K}_T^\ell$ and $\mathcal{Q}_T^\ell$, respectively, must be solved. 
Although these problems are local, they remain infinite-dimensional and therefore require a fine-scale discretization. 
To this end, we introduce a fine mesh $\mathcal{T}_h$ obtained by refining the coarse mesh~$\mathcal{T}_H$. 
The fine mesh must be sufficiently fine to resolve all microscopic features of the coefficients.  
The problems \cref{eq:defKTell,eq:defQTell} are then solved on local $\mathcal{P}^1$-finite element spaces defined on submeshes of $\mathcal{T}_h$.  
In the fully discrete convergence analysis, the continuous space $V$ is replaced by the fine-scale finite element space, and most arguments carry over directly; see, e.g., \cite[Ch.~4.4]{Mlqvist2020}. 
This yields an a priori error estimate for the fully discrete LOD approximation with respect to the fine-scale finite element solution, analogous to \cref{thm:convergenceloc}. 
An estimate with respect to the solution of the original PDE then follows by applying the triangle inequality together with standard finite element approximation results.

\subsection{Numerical experiments}

\begin{figure}
    \centering
        \includegraphics[width = .4\linewidth,height=0.365\linewidth]{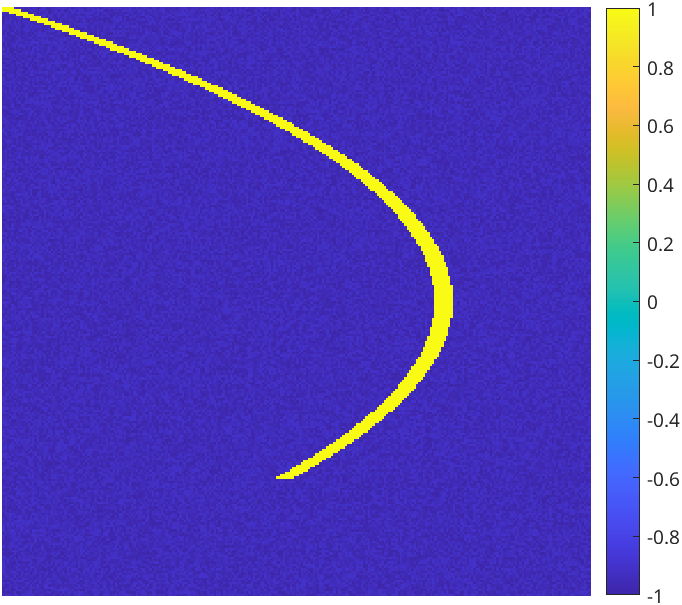}
       \hspace{.5cm}
 \includegraphics[width = .4\linewidth,height=0.365\linewidth]{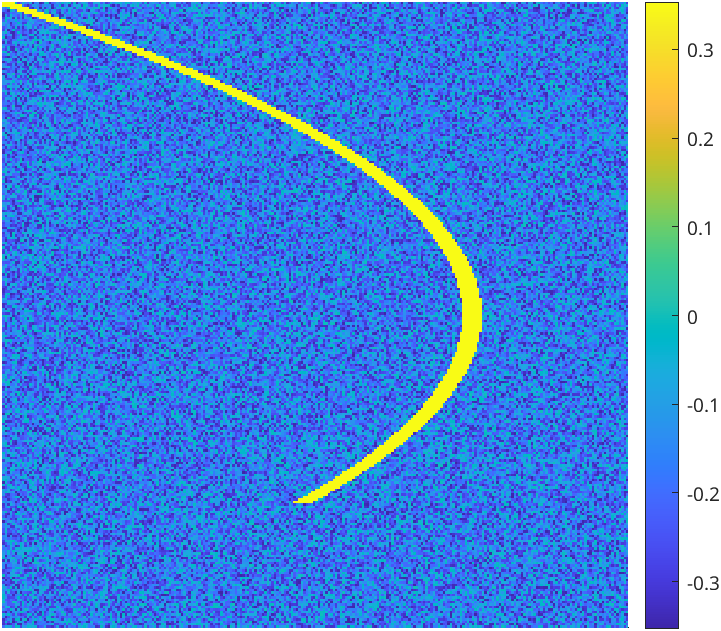}
	\caption{Illustration of the multiscale coefficients $a_{12}$ (left) and $b_1$ (right) used in the numerical experiment.}\label{fig:coeff}
\end{figure}

For the numerical experiments, we consider the domain $\Omega := (0,1)^2$ and choose the stabilization parameter $\sigma \coloneqq  1$ in~\eqref{bilin form a}. 
In this setting, the Poincaré constant~\cref{Cp} for the space $V$ is given by $C_\mathrm{P} = \pi$. 
Therefore, the admissible range of $\delta$ in the Cordes-type condition \cref{Cordes} is $\delta \in (\frac{1}{1+\pi^2}, 1]$.

We choose the data of the numerical experiment as
\begin{align*}
	A(x,y) &\coloneqq \begin{pmatrix}
		\tfrac{\sqrt{11}}{4} & a_{12}(x,y) \\
		a_{12}(x,y) & \tfrac{3\sqrt{11}}{4}
	\end{pmatrix}, \quad 
	b(x,y) \coloneqq \begin{pmatrix}
		b_1(x,y) \\ 1
	\end{pmatrix}, \quad 
	f(x,y) \coloneqq \cos(\pi x) y^3
\end{align*}
for $(x,y)\in \Omega$, where $a_{12}$ is a realization of a background random field that is piecewise constant on a Cartesian mesh with mesh size $\epsilon = {2^{-8}}$, with independent and identically distributed values in the interval $[-1,-0.9]$, augmented by a channel taking the constant value~$1$. This channel consists of the elements whose midpoints have a distance less than $4\epsilon$ from a parabola, the corresponding element values are set to~1. Analogously, we define $b_1$, with realizations in the interval $[-1/{\sqrt{8}},0]$, augmented by a channel taking the value $1/{\sqrt{8}}$. For an illustration of the multiscale coefficients $a_{12}$ and $b_1$, we refer to  \cref{fig:coeff}. 
It is quickly checked that the pair~$(A,b)$ satisfies the Cordes-type condition \cref{Cordes} with $\delta = \frac{1}{10} \in \big(\tfrac{1}{1+\pi^2},1\big]$.

For the multiscale approximation, we construct a hierarchy of meshes by uniformly refining the initial mesh shown in \cref{fig:nonpp} (left). 
For simplicity, we denote the meshes in this hierarchy by $\mathcal{T}_{2^0}, \mathcal{T}_{2^{-1}}, \dots, \mathcal{T}_{2^{-6}}$, where the subscript indicates the side length of the squares formed by joining opposing triangles. 
For the fine-scale discretization, we use the mesh $\mathcal{T}_{2^{-9}}$, and the corresponding fine-scale approximation is denoted by $z_h$. 
In the following, we study the approximation errors of the fine-scale discretized counterparts of the original and post-processed multiscale approximations defined in \cref{eq:locmethod,eq:postprocessed}, denoted by $\tilde z_{H,h}^\ell$ and $\hat z_{H,h}^\ell$, respectively.  
Specifically, we consider the energy errors
\begin{align*}
	\mathrm{err}_{\fraka}(H,\ell) &\coloneqq \|z_h - \tilde z_{H,h}^\ell\|_\fraka, \qquad 
	\mathrm{err}_{\fraka}^\mathrm{pp}(H,\ell)\coloneqq \|z_h - \hat z_{H,h}^\ell\|_\fraka,
\end{align*}
where $\|\cdot\|_\fraka$ denotes the energy norm induced by the bilinear form $\fraka$. We recall that, in view of \cref{Lmm: char of grad u}, the energy norm is equivalent to the norm $\|D(\cdot)\|_{L^2}$.

For the multiscale approximation without post-processing, as defined in \cref{eq:locmethod}, \cref{fig:nonpp} shows that the error decreases with mesh refinement, but the convergence is very slow, with estimated convergence rates well below one. 
This behavior is due to the low regularity of the right-hand side of the variational problem \cref{z prob}, which lies only in the dual space $V^*$. 
Consequently, no rigorous convergence rates can be extracted; see \cref{Rk: nec corr}.

However, post-processing as introduced in \cref{eq:postprocessed} changes the picture completely. 
This post-processing allows one to extract convergence rates from the right-hand side despite the low regularity of the corresponding functional. 
This is clearly visible in \cref{fig:pp}, where, for sufficiently large oversampling parameters, second- and third-order convergence is observed for the post-processed approximations with $p = 1$ and $p = 2$, respectively. For the post-processing, we choose the operator $\Pi_H$ in \cref{eq:approxf} as the nodal interpolation onto first-order ($p = 1$) and second-order ($p = 2$) Lagrange finite element spaces defined on the mesh $\mathcal{T}_H$. These numerical observations are in line with the theoretical predictions from \cref{thm:convergenceloc}. 
Note that, for fixed oversampling parameters $\ell$, the errors initially decrease with optimal order but eventually reach a plateau, the value of which depends on the specific choice of~$\ell$. 
This behavior is consistent with the theoretical prediction in \cref{thm:convergenceloc} and contrasts with the results in \cite{FGP24}, where, for fixed oversampling parameters, the error initially decreases but then increases sharply. 
The increase is caused by a negative power of $H$ multiplying the (exponentially decaying) localization error in the error estimate. Here, this issue is avoided by adapting the localization strategy proposed in \cite{Hauck2026}.

We emphasize that, by employing this new post-processing, convergence can be shown in an $H^2$-type norm of the solution $u$ which was not possible in previous work; see \cite{FGP24}.

\begin{figure}
    \centering
    
    \includegraphics[height=0.425\linewidth]{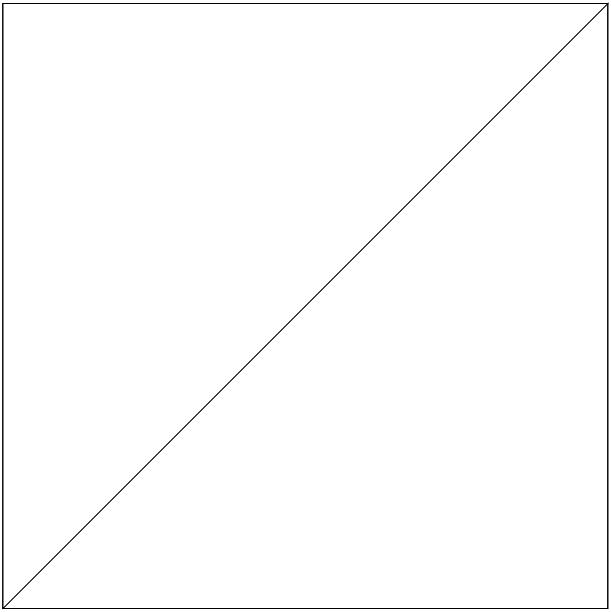}
           \hspace{.5cm}
\includegraphics[width=0.45\linewidth]{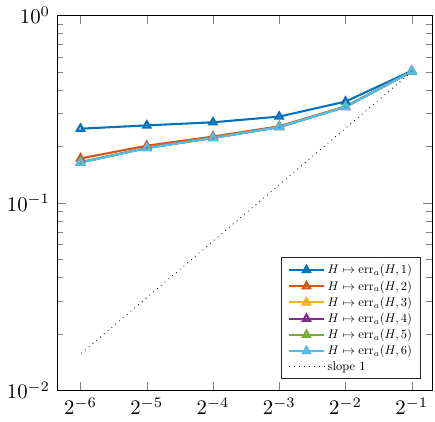}
\caption{Initial mesh for the mesh generation (left). Energy errors of the multiscale approximation $\tilde z_{H,h}^\ell$ (without post-processing) for different choices of the oversampling parameter $\ell$, plotted as a function of the coarse mesh size $H$ (right).}
    \label{fig:nonpp}
\end{figure}
\begin{figure}
    \centering
    \includegraphics[width=0.45\linewidth]{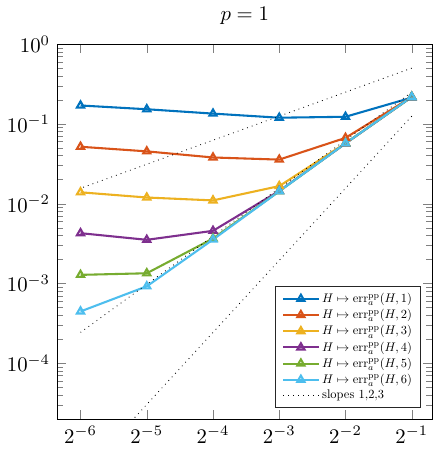}\hfill
\includegraphics[width=0.45\linewidth]{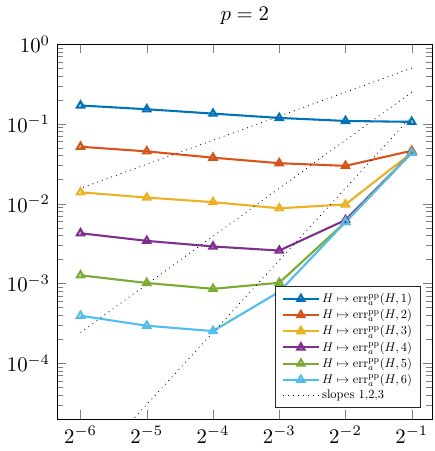}
\caption{Energy errors of the post-processed multiscale approximation $\hat z_{H,h}^\ell$ for different choices of the oversampling parameter $\ell$, plotted as a function of the coarse mesh size $H$. Two choices for the polynomial degree used for the post-processing (cf.~\cref{eq:idealmod} and \cref{eq:postprocessed}) are considered: $p = 1$ (left) and $p = 2$ (right).} \label{fig:pp}
\end{figure}

\subsection*{Acknowledgments}
The authors are grateful to D.~Gallistl (Friedrich-Schiller-Universit\"{a}t Jena) for providing a monoscale code that served as the foundation for the implementation developed in this work. M.~Hauck and R.~Maier acknowledge funding from the Deut\-sche Forschungsgemeinschaft (DFG, German Research Foundation) -- Project-ID 258734477 -- SFB 1173. 
Parts of this work were conducted during M.~Hauck's and R.~Maier's stay at the
Hausdorff Research Institute for Mathematics funded by the Deutsche Forschungsgemeinschaft (DFG, German Research Foundation) under Germany's Excellence Strategy – EXC-2047/2 – 390685813.

\bibliographystyle{alpha}
\bibliography{references}

\end{document}